\documentclass[11 pt,reqno]{amsart}

\usepackage{amsmath,amsthm,amssymb,amscd,graphicx}
\usepackage{bbm, dsfont}
\usepackage{color}
\usepackage{hyperref}
\usepackage{enumerate}
\usepackage{mathtools}
\usepackage[marginparwidth=2cm]{geometry}
\usepackage{mathrsfs}
\usepackage{multirow}

\newtheorem{thm}{Theorem}[section]
 
 \newtheorem{lem}[thm]{Lemma}
 
 \newtheorem{prop}[thm]{Proposition}
 
 \theoremstyle{definition}
 
 \theoremstyle{remark}
 \newtheorem{rem}[thm]{Remark}
 
 \numberwithin{equation}{section}

\def\be#1 {\begin{equation} \label{#1}}
\newcommand{\ee}{\end{equation}}
\renewcommand{\phi}{\varphi}

\def\C{\mathbb C}
\def\R{\mathbb R}

\def\HH{\mathbb H}
\def\N{\mathbb N}
\def\E{\mathcal E}

\def\e{e}

\def\pa{\partial}
\def\eps{\epsilon}
\def\dis{\displaystyle}    

\definecolor{gr}{rgb}   {0.,   0.69,   0.23 }
\definecolor{bl}{rgb}   {0.,   0.5,   1. }
\definecolor{mg}{rgb}   {0.85,  0.,    0.85}
\definecolor{yl}{rgb}   {0.8,  0.7,   0.}
\definecolor{or}{rgb}  {0.7,0.2,0.2}

\newcommand{\wt}{\widetilde}

\catcode`\ = \active\def {\'e} \catcode`\ = \active\def {\'E}  \catcode`\Ë = \active\def Ë{\`A}\catcode`\ = \active\def {\`e}
\catcode`\ = \active\def {\c{c}} \catcode`\ = \active\def
{\`a} \catcode`\ = \active \def {\^e} \catcode`\ = \active
\def {\`u} \catcode`\ = \active \def {\^o} \catcode`\ =
\active \def {\^u} \catcode`\ = \active \def {\^{\i}}
\catcode`\ = \active \def {\^a} \catcode`\ = \active \def
{{\"o}}

\renewcommand{\Re}{  {\mathfrak{Re}}  }
 \renewcommand{\Im}{   {\mathfrak{Im}} }
\newcommand{\ov}{  \overline  }

\newcommand\<{\langle}
\renewcommand\>{\rangle}

\begin{document}

\thanks{The author is supported by the grants   "BEKAM''  ANR-15-CE40-0001 and  "ISDEEC'' ANR-16-CE40-0013}

%%%%%%%%
%\author{ Laurent Thomann }
\address{Universit de Lorraine, CNRS, IECL, F-54000 Nancy, France}
\email{laurent.thomann@univ-lorraine.fr}

\title[On multi-solitons for coupled LLL equations]{On multi-solitons for coupled Lowest Landau Level equations}

\subjclass[2000]{35Q55 ; 37K05 ; 35C07 ; 35B08}    

\keywords{Nonlinear Schr\"odinger equation, Lowest Landau Level, stationary solutions, progressive waves, multi-solitons, growth of Sobolev norms.}

\begin{abstract} We consider a coupled system of nonlinear Lowest Landau Level equations. We first show the  existence of multi-solitons with an exponentially localised error term in space, and  then we prove a uniqueness result. We also show a long time stability result of the sum of  traveling waves having all the same speed, under the condition that they are localised far away enough from each other. Finally, we observe that these multi-solitons provide examples of dynamics for the linear Schr\"odinger equation with harmonic potential perturbed by a time-dependent potential.
 \end{abstract}

\maketitle

\section{Introduction and main results}

In this paper, we continue the study of a system of coupled Lowest Landau  Level (LLL) equations which was initiated in \cite{Schw-Tho}. Denote by $\E$   the   Bargmann-Fock space defined as 
$$
\mathcal{E} =\big \{\, u(z) = e^{-\frac{|z|^2}{2}} f(z)\,,\;f \; \mbox{entire\ holomorphic}\,\big\}\cap L^2(\C )
$$
and consider  $\Pi$  the     orthogonal projection on $\mathcal{E}$. The LLL system then reads
\begin{equation}\label{eq0} 
\left\{
\begin{aligned}
&i\partial_{t}u= \Pi (|v|^2 u), \quad   (t,z)\in \R\times \C,\\
&i\partial_{t}v=  \sigma \Pi (|u|^2 v),\\
&u(0,\cdot)=  u_0 \in\E,\; v(0,\cdot)=  v_0 \in\E,
\end{aligned}
\right.
\end{equation}
where $\sigma \in \{1,-1\}$ is fixed.  Such systems arise in the description of fast rotating Bose-Einstein condensates in interaction: for more details and references on the modeling, see~\cite{ABD, Ho, Mueller-Ho}, the introduction of \cite{GGT}, and references therein. The system \eqref{eq0} is Hamiltonian with the structure
\begin{equation*} 
\left\{
\begin{aligned}
&\dot u=-i\frac{\delta \mathcal{H}}{\delta \ov u},& \quad &\dot{\ov u}=i\frac{\delta \mathcal{H}}{\delta u},\\
&\dot v=-i\sigma  \frac{\delta \mathcal{H}}{\delta \ov v},& \quad &\dot{\ov v}=i\sigma \frac{\delta \mathcal{H}}{\delta v},
\end{aligned}
\right.
\end{equation*} 
where the Hamiltonian functional is given by 
$$\mathcal{H}(u,v)=\int_{\C}|u|^2|v|^2dL,$$
and where $L$ stands for Lebesgue measure on $\C$. For  mathematical results on LLL equations we refer to \cite{Nier, ABN, BiBiCrEv2, GGT,Clerck-Evnin}. \medskip

In  the case~$\sigma=-1$, we have constructed in~\cite{Schw-Tho} traveling-waves (solitons) solutions to \eqref{eq0} and the aim of the present work is to show the existence of multi-solitons and study some of their properties. When $\sigma=1$, such solutions are excluded, because their existence  would contradict the conservation laws of the system (see~\cite[Proposition~1.4]{Schw-Tho}). Therefore, from now on, we assume that $\sigma=-1$ and we consider the system 
 \begin{equation}\label{eq2} 
\left\{
\begin{aligned}
&i\partial_{t}u= \Pi (|v|^2 u), \quad   (t,z)\in \R\times \C,\\
&i\partial_{t}v=  - \Pi (|u|^2 v),\\
&u(0,\cdot)=  u_0 \in\E,\; v(0,\cdot)=  v_0 \in\E.
\end{aligned}
\right.
\end{equation}

There are many  results concerning the existence of multi-solitons for dispersive equations (including Korteweg-de Vries, Schr\"odinger, and wave equations) and we refer to the survey \cite{Martel} for references on the subject. More precisely, regarding the construction of multi-solitons for the nonlinear Schr\"odinger equations we address to the works \cite{Martel-Merle, CoteLeCoz,CoteMarMer} and to the  recent survey~\cite{LCT}. In \cite{KMR, Martel-Raphael}, the authors study strong interactions of solitons. We also mention the articles~\cite{Ianni-LeCoz, DeleCozWeis}  in which solitary waves with different speeds are constructed for Schr\"odinger systems.

  \subsection{Symmetries and conservation laws}
The system \eqref{eq2} is preserved by several symmetries, which induce conservation laws (see \cite[Section 2]{GGT} for more details). These symmetries are phase rotations
\begin{equation*} 
T_{\theta_1,\theta_2} : (u,v)(z) \mapsto \big(e^{i \theta_1 }u(z), e^{i \theta_2 }v(z)\big)  \qquad \mbox{for $(\theta_1, \theta_2)  \in \mathbb{T}^2$},
\end{equation*}
 space rotations 
\begin{equation*} 
L_{\theta} :(u,v)(z)  \mapsto \big(u(e^{i\theta}z), v(e^{i\theta}z) \big)  \qquad \mbox{for $\theta \in \mathbb{T}$},
\end{equation*}
and magnetic translations 
\begin{equation*} 
R_{\alpha} :(u,v)(z)  \mapsto \big(u(z+\alpha) e^{\frac{1}{2}(\overline z \alpha - z \overline{\alpha})}, v(z+\alpha) e^{\frac{1}{2}(\overline z \alpha - z \overline{\alpha})} \big) \qquad \mbox{for $\alpha \in \mathbb{C}$}.
\end{equation*}
The corresponding conservation laws are: the mass
$$ M(u)=\int_{\C}|u(z)|^2dL(z), \quad M(v)=\int_{\C}|v(z)|^2dL(z),$$
  the angular momentum
    \begin{equation*}
 P_-(u,v)=\int_{\C}\big(|z|^2-1\big)\big(|u(z)|^2-|v(z)|^2\big)dL(z),
 \end{equation*}
and the magnetic momentum
   \begin{equation*}
 Q_-(u,v)=\int_{\C}z\big(|u(z)|^2-|v(z)|^2\big)dL(z).
 \end{equation*}

\subsection{Functional spaces} In order to state our results we need to define a few spaces. Namely, for $s \geq 0$, we denote by 
$$L^{2,s}=\big\{u\in \mathscr{S}'(\C), \;\<z\>^su \in L^2(\C)\big\}, \quad \<z\>=(1+|z|^2)^{1/2}$$
 the weighted Lebesgue space and we define
 $$L^{2,s}_{\E}=L^{2,s}  \cap {\E}.$$
 It turns out that this latter space coincides with the harmonic Sobolev space. For  $s \geq 0$ we consider    
\begin{equation} \label{def-sobo}
\HH^{s}(\C) = \big\{ u\in \mathscr{S}'(\C),\; {H}^{s/2}u\in L^2(\C)\big\}\cap {\E},
\end{equation}
equipped with the natural norm $ \|u\|_{\HH^s(\C)} =\|H^{s/2} u\|_{L^2(\C)}$. Then, we have $\HH^{s}(\C)  = L^{2,s}_{\E}$ and the following equivalence of norms holds true
\begin{equation} \label{equiNor}
c\|\<z\>^s u\|_{L^2(\C)} \leq  \|u\|_{\HH^s(\C)} \leq C\|\<z\>^s u\|_{L^2(\C)}, \quad \forall\,u \in L^{2,s}_{\E},
\end{equation}
see~\cite[Lemma~C.1]{GGT} for a proof.  \medskip

Similarly, for $\kappa \geq0$, we denote by 
$$\mathcal{X}^{\kappa} =\big\{u\in \mathscr{S}'(\C), \; e^{ \kappa |z|}u \in L^2(\C)\big\} ,$$
and we set 
$$\mathcal{X}^{\kappa}_\mathcal{E}=\big\{u\in \mathscr{S}'(\C), \; e^{ \kappa |z|}u \in L^2(\C)\big\} \cap {\E}.$$
\subsection{Global existence results for the  system \eqref{eq2}}

We first recall the  global well-posedness result  for \eqref{eq2}, which is  contained in~\cite[Theorem~1.1]{Schw-Tho}.

\begin{thm}[Theorem~1.1, \cite{Schw-Tho}]\label{thmCauchy}
For every $(u_0,v_0)\in \mathcal{E} \times \mathcal{E}$, there exists a unique solution $(u,v)\in \mathcal{C}^{\infty}  (\R ,\mathcal{E}\times \mathcal{E})$ to the system  \eqref{eq2}, and this solution depends smoothly on $(u_0,v_0)$. Moreover,
\begin{enumerate}[$(i)$]
\item  for every  $t\in \R$ 
$$ M(u)=\int_{\C}|u(t,z)|^2dL(z)=M(u_0), \quad M(v)=\int_{\C}|v(t,z)|^2dL(z)=M(v_0),$$
and 
 $$\mathcal{H}(u,v)=  \int_{\C}|u(t,z)|^2|v(t,z)|^2dL(z)=\mathcal{H}(u_0,v_0) \,;$$
\item   if $(zu_0,zv_0)\in L^2(\C )\times  L^2(\C )$, then  $\big(zu(t), zv(t)\big)\in L^2(\C)\times  L^2(\C )$ for every $t\in \R$, and 
      \begin{equation*}
 P_-(u,v)=\int_{\C}\big(|z|^2-1\big)\big(|u(t,z)|^2-|v(t,z)|^2\big)dL(z)= P_-(u_0,v_0),
 \end{equation*}
   \begin{equation*}
 Q_-(u,v)=\int_{\C}z\big(|u(t,z)|^2-|v(t,z)|^2\big)dL(z)= Q_-(u_0,v_0)\,;
 \end{equation*}
\item if for some $s>0$, $\big(\langle z\rangle ^s u_0, \<z\rangle ^s v_0\big)\in L^2(\C)  \times  L^2(\C ) $, then  $\big(\langle z\rangle ^su(t), \langle z\rangle ^sv(t)\big) \in L^2(\C) \times  L^2(\C ) $ for every $t\in \R$.
 \end{enumerate}
\end{thm}
 
 We can also prove polynomial  bounds on the possible growth of Sobolev norms for \eqref{eq2},  we refer to~\cite[Theorem~1.5]{Schw-Tho} for details. \medskip
 
It turns out that equation \eqref{eq2} is also globally well-posed for exponentially localised functions and we are able to obtain a quantitative estimate on the long time behaviour of the solutions as well as a stability result.
\begin{prop}\label{propExp}
Let $\kappa \geq 0$, then the following properties hold true:
\begin{enumerate}[$(i)$]
\item assume that  $( u_0,    v_0)\in \mathcal{X}^{\kappa}_\mathcal{E} \times \mathcal{X}^{\kappa}_\mathcal{E}$, then  the corresponding solution to \eqref{eq2} satisfies $  (u,   v) \in\mathcal{C}^{\infty}  \big(\R , \mathcal{X}^{\kappa}_\mathcal{E} \times \mathcal{X}^{\kappa}_\mathcal{E}\big)$. Moreover,    for every $t\in \R$, 
\begin{equation}\label{mexp}
\begin{aligned} 
&  \| e^{\kappa|z|} u(t)\|_{L^2(\C)}    \leq \| e^{\kappa|z|} u_0\|_{L^2(\C)}e^{c_\kappa  \|  v_0\|^2_{L^2} |t|}  \\[4pt]
&   \| e^{\kappa|z|} v(t)\|_{L^2(\C)}    \leq \| e^{\kappa|z|} v_0\|_{L^2(\C)}e^{c_\kappa  \|  u_0\|^2_{L^2} |t|}  \;,
\end{aligned}
\end{equation}
where the constant $c_\kappa>0$ only depends on $\kappa>0$ (notice that $c_0=0$ by the conservation of the $L^2-$norm);
\item   consider two solutions $(u,   v) \in\mathcal{C}^{\infty}  \big(\R , \mathcal{X}^{\kappa}_\mathcal{E} \times \mathcal{X}^{\kappa}_\mathcal{E}\big)$ and $(\wt{u},  \wt{v}) \in\mathcal{C}^{\infty}  \big(\R , \mathcal{X}^{\kappa}_\mathcal{E} \times \mathcal{X}^{\kappa}_\mathcal{E}\big)$ to \eqref{eq2}. Then, for all $t\in \R$
\begin{multline} \label{mexp2}
  \| e^{\kappa|z|} \big(u(t)-\wt{u}(t)\big)\|^2_{L^2(\C)}  + \| e^{\kappa|z|} \big(v(t)-\wt{v}(t)\big)\|^2_{L^2(\C)}  \\
    \leq\big( \| e^{\kappa|z|} \big(u_0-\wt{u}_0\big)\|^2_{L^2(\C)}+   \| e^{\kappa|z|} \big(v_0-\wt{v}_0\big)\|^2_{L^2(\C)} \big)   e^{c_\kappa  (\|  u_0\|^2_{L^2}+ \|  \wt{u}_0\|^2_{L^2}+\|  v_0\|^2_{L^2}+ \|  \wt{v}_0\|^2_{L^2}) |t|} ,
\end{multline}
where the constant $c_\kappa>0$ only depends on $\kappa>0$.
\end{enumerate}
\end{prop}

The estimate \eqref{mexp} is sharp, see~\eqref{sharp-exp} below.

  \subsection{Solitons and multi-solitons}
Using the  invariances induced  by phase rotations and magnetic translations, it is natural to look for   particular solutions  for equation~\eqref{eq2}  of the form
 \begin{equation}\label{trav}
 \big(u(t,z), v(t,z)\big)=\big(e^{-i\lambda t}U(z+\alpha t) e^{\frac{1}{2}(\overline z \alpha - z \overline{\alpha})t},e^{-i \mu t}V(z+\alpha t) e^{\frac{1}{2}(\overline z \alpha - z \overline{\alpha})t} \big),
 \end{equation}
 that we call progressive  or traveling waves. Such solutions do exist, and by~\cite[Theorem~1.6]{Schw-Tho}, the progressive waves in $\mathcal{E}$, when $\alpha \neq 0$, which have a finite number of zeros are given by the initial conditions 
  
 \begin{equation}\label{prog}
\left\{
\begin{aligned}
& U= Ke^{ia} \big(\frac12 \phi_0^\gamma +\frac{\sqrt{3}}2 ie^{i\theta} \phi_{1}^\gamma      \big)   \\
& V= Ke^{i b} \big(\frac12 \phi_0^\gamma -\frac{\sqrt{3}}2i e^{i\theta} \phi_{1}^\gamma      \big),
\end{aligned}
\right.
\end{equation}
with $\gamma \in \C$ and
\begin{equation*}
\phi_n^\gamma(z)  = \frac{1}{\sqrt{\pi n!}} (z-\overline \gamma)^n e^{-\frac{|z|^2}{2}-\frac{|\gamma|^2}{2} + \gamma z}.
\end{equation*} 
with $K \geq 0  $, with $\theta, a,b \in  \R$,  where
 \begin{equation}\label{def-para1}
  \lambda = \frac{K^2}{32\pi}(7+2\sqrt{3}\Im\big(\gamma e^{-i\theta})\big), \quad \mu = \frac{K^2}{32\pi}\big(-7+2\sqrt{3}\Im(\gamma e^{-i\theta})\big),
 \end{equation}
 and with the speed 
  \begin{equation}\label{def-para2}
\alpha = \frac{\sqrt{3}}{32\pi}K^2e^{-i\theta}.
 \end{equation}

It is interesting to notice that any non trivial    traveling wave of the form~\eqref{trav} has growing Sobolev norms. Actually,  if $u(t)= e^{-i\lambda t}R_{\alpha t}U $, then   
 \begin{equation*}  
 \|\<z\>^s u(t)\|_{L^{2}(\C)} =\|\<z\>^s R_{\alpha t}U\|_{L^{2}(\C)}= \|\<z-\alpha t \>^s U\|_{L^{2}(\C)} \sim |\alpha|^s |t|^s  \|U\|_{L^{2}(\C)},
   \end{equation*} 
when $t \longrightarrow \pm \infty$. Moreover, the previous growth of norms is the strongest possible by~\cite[Theorem~1.5]{Schw-Tho}. Similarly, when $t \longrightarrow \pm \infty$,
 \begin{equation} \label{sharp-exp}
 \| e^{\kappa|z|} u(t)\|_{L^{2}(\C)} =\|e^{ \kappa |z-\alpha t |} U\|_{L^{2}(\C)} \sim  e^{ \kappa |\alpha ||t |}  \big\| e^{-\kappa \tau \Re(ze^{-i \theta})}U\big\|_{L^{2}(\C)},
   \end{equation} 
with $\theta=\arg(\alpha)$ and $\tau=\text{sign}(t)$. Thus \eqref{sharp-exp}     shows the sharpness of \eqref{mexp}.

\subsubsection{Existence of multi-solitons}
A natural question is the existence of solutions to \eqref{eq2} which are a finite sum of such traveling  waves. The answer is positive and this is the content of the following result : 

\begin{thm}\label{thm-multi}
Let $n \geq 1$. For $1 \leq j \leq n$, let  $(K_j, a_j, b_j, \theta_j, \gamma_j ) \in \R^*_+ \times \R \times  \R \times \R \times \C$ and consider the parameters $(\lambda_j, \mu_j,\alpha_j) \in \R \times \R \times \C^*$  given by~\eqref{def-para1} and~\eqref{def-para2}. Assume that $\alpha_j \neq \alpha_{\ell}$ for $j \neq \ell$. Denote by 
$$\alpha_\sharp = \min_{j \neq \ell}{|\alpha_j-\alpha_\ell|}.$$
Then, for all $\kappa>0$, there exists a solution $(u,v)\in \mathcal{C}^{\infty}  \big(\R , \mathcal{X}^{\kappa}_\mathcal{E} \times \mathcal{X}^{\kappa}_\mathcal{E}\big)$ to equation~\eqref{eq2} of the form 
 \begin{equation}\label{sol-thm}
\left\{
\begin{aligned}
&u(t,z)=\sum_{j=1}^n e^{-i\lambda_j t}U_j(z+\alpha_j t) e^{\frac{1}{2}(\overline z \alpha_j - z \overline{\alpha_j})t}+r_1(t,z) \\
& v(t,z)=\sum_{j=1}^n   e^{-i \mu_j t}V_j(z+\alpha_j t) e^{\frac{1}{2}(\overline z \alpha_j - z \overline{\alpha_j})t}  +r_2(t,z) ,
\end{aligned}
\right.
\end{equation}
where the $(U_j,V_j)$ take the form \eqref{prog} and where the error terms satisfy : for all 
$$c<\frac 14 $$
and all $m \in \N$, there exists $C_{m,\kappa}>0$ such that for all $t \geq 0$
  \begin{equation} \label{bb}
 \big \|e^{\kappa|z|} (\partial^m_t r_1)(t)  \big\|_{L^2}+  \big \|e^{\kappa|z|} (\partial^m_t r_2)(t)  \big\|_{L^2} \leq C_{m,\kappa} e^{-c \alpha^2_\sharp t^2}.
  \end{equation}
\end{thm}

Notice that thanks to the Carlen inequality~\eqref{hyp} below, the bound~\eqref{bb} implies the following pointwise  estimate : for all $c<1/4$, all $m \in \N$ and all $z\in \C$
    \begin{equation*} 
 \big |(\partial^m_t  r_1)(t,z)  \big|   +    \big|(\partial^m_t  r_2)(t,z)  \big|  \leq C_{m,\kappa} e^{-c\alpha^2_\sharp t^2} e^{-\kappa|z|}.
  \end{equation*}~
  
The construction of multi-solitons for \eqref{eq2} relies on classical arguments, including  backwards in time integration and energy estimates. We refer to~\cite{Martel-Merle, CoteLeCoz, CoteMarMer, Fer2,Faou-Rapha} where these methods were used. The situation here  is very favorable since  in the space~$\E$, any~$L^p$  norm ($p\geq 2$) can be controlled (see \eqref{hyp}), namely
  \begin{equation*}
  \|u\|_{L^\infty(\C)} \leq C \|u\|_{L^2(\C)} , \quad \forall u \in \E.  
  \end{equation*}
In particular, this allows to prove that  the system~\eqref{eq2}  is globally well-posed in $\E$ and to close energy estimates in $\E$.\medskip

  In \eqref{bb} we observe that  the decay depends only on $\alpha_\sharp$ and not on the frequencies $\lambda_j$ (resp.~$\mu_j$) of the traveling waves. This decay is induced by the Gaussian nature of  the traveling waves. Such a phenomenon is in contrast with NLS, where the solitons have an exponential decay and where the speed of convergence depends on the frequencies of the solitons \cite{CoteLeCoz}.    The same rate of decay as in~\eqref{bb} is obtained in~\cite{Fer2} where multi-Gaussian solutions are constructed for the Schr\"odinger equation with logarithmic nonlinearity (logNLS).  Another interesting similarity with the results in~\cite{Fer2}, is that the convergence to the multi-soliton holds in weighted Sobolev spaces (namely in $H^1 \cap \mathcal{F}(H^1)$). In the present case, one can even upgrade to exponential weights, and this is due to the absence of linear part in the equation \eqref{eq2} (see Remark~\ref{rem14} for the case of LLL with a linear part). We refer to \cite{Car-Galla} and references therein for more results on the dynamics of  logNLS.\medskip

The result of Theorem~\ref{thm-multi} actually holds under the weaker assumption that each   traveling wave $(U,V) \in \E \times \E$ of the sum~\eqref{sol-thm}  satisfies a Gaussian bound 
\begin{equation}\label{gb}
|U(z)|+|V(z)| \leq C e^{-c_0 |z|^2},
\end{equation}
for some $C, c_0 >0$, and the proof of the Theorem~\ref{thm-multi} is written using only the assumption~\eqref{gb}. In this latter case, \eqref{bb} is replaced by 
  \begin{equation} \label{cc}
 \big \|e^{\kappa|z|} (\partial^m_t r_1)(t)  \big\|_{L^2}+  \big \|e^{\kappa|z|} (\partial^m_t r_2)(t)  \big\|_{L^2} \leq C_{m,\kappa} e^{-\wt{c_m} \alpha^2_\sharp t^2},
  \end{equation}
for some $\wt{c_m}>0$. However, we do not know if there exist other  traveling waves (with $\alpha \neq 0$) than the ones exhibited in~\eqref{prog} (such traveling waves would then have an infinite number of zeros by~\cite[Theorem~1.6]{Schw-Tho}). \medskip

In the hypotheses of Theorem~\ref{thm-multi}, one can also allow for the case where $\alpha_j=0$ for at most only one  $1\leq j\leq n$. In this case, $(e^{-i\lambda t}U(z) , e^{-i\mu t}V(z))$ is a solution to \eqref{eq2} if and only if  
 \begin{equation} \label{sysM}
\left\{
\begin{aligned}
& \lambda U= \Pi(|V|^2 U)     \\
& \mu V= -\Pi(|U|^2 V). 
\end{aligned}
\right.
\end{equation}
By Theorem~\ref{thmDec}, any solution   $(U,V) \in \E \times \E$  to \eqref{sysM} satisfies the bound \eqref{gb} for all $c_0<1/2$. Examples of solutions of~\eqref{sysM}  are for instance :
\begin{itemize}
\item $\dis (U,V)=  (A_1\phi_{n_1}^{\gamma}, A_2\phi_{n_2}^{\gamma})$, for any $A_1,A_2,\gamma \in \C  $ and $n_1,n_2 \in  \N$, by~\cite[Theorem~1.6]{Schw-Tho}~; \vspace{4pt}
\item $(U,V)= (U,U)$ and $\mu=-\lambda$ where $U \in \E$ is any  solution of $\dis   \lambda U= \Pi(|U|^2 U) $. We refer to~\cite[Appendix~A]{GGT} for explicit examples.
\end{itemize}
 \medskip

By reversibility of the equation~\eqref{eq2}, similar multi-solitons can be constructed in the regime ${t\longrightarrow - \infty}$. Actually, if $(u,v)$ is a solution to \eqref{eq2}, then $(\wt{u}, \wt{v})$ is also a solution where $(\wt{u}, \wt{v})(t) := (v,u)(-t)$. However, the question whether there exists $(r_1,r_2)$ such that \eqref{bb} holds for all $t \in \R$ is left open. \medskip

Since the terms in \eqref{sol-thm} decouple when   $t \longrightarrow +\infty$, it is easy to observe   that the solutions of Theorem~\ref{thm-multi} satisfy
 \begin{equation}\label{form-cons}
\left.
\begin{array}{lll}
&\dis M(u)= M(v)= \sum_{j=1}^n K_j^2, & \dis \mathcal{H}(u,v)=  \frac{11}{64 \pi}  \sum_{j=1}^n K_j^2,  \vspace{5pt} \\
& \dis  P_-(u,v)=\sqrt{3}\sum_{j=1}^n \Im(\gamma_j e^{-i\theta_j})K_j^2, \qquad & \dis Q_-(u,v)=-\frac{\sqrt{3}}2i     \sum_{j=1}^n  e^{-i\theta_j}K_j^2.
\end{array}
\right.
\end{equation}

\begin{rem}\label{rem14}
 We can also construct multi-solitons  for the system 
 \begin{equation}\label{sys-conj}
\left\{
\begin{aligned}
&i\partial_{t}\wt{u}- \delta H\wt{u} =  \Pi (|\wt{v}|^2 \wt{u}), \quad   (t,z)\in \R\times \C,\\
&i\partial_{t}\wt{v}-\delta  H\wt{v}=  -   \Pi (|\wt{u}|^2 \wt{v}),\\
&\wt{u}(0,z)=  u_0(z),\; \wt{v}(0,z)=  v_0(z),
\end{aligned}
\right.
\end{equation}
where $\delta \in \R$ is a   given dispersion parameter. Actually,    the change of unknown $(\wt{u}, \wt{v}) =e^{-i \delta t H}(u,v)$ shows that  the system~\eqref{eq2} is equivalent to~\eqref{sys-conj} (see \cite[Section 1.7.2]{Schw-Tho} for more details). Recall that $e^{i \tau H}= e^{2i\tau }L_{2\tau}$ (which can be directly checked by testing on the complete  family $(\phi_n)_{n\geq 0}$), then Theorem~\ref{thm-multi} enables the  construction of   the following multi-solitons  for \eqref{sys-conj}
 \begin{equation*} 
\left\{
\begin{aligned}
&\wt{u}(t,z)=\sum_{j=1}^n e^{-i(\lambda_j +2\delta)t}      L_{-2\delta t}U_j(z+\alpha_j t) e^{\frac{1}{2}(\overline z \alpha_j - z \overline{\alpha_j})t}+\wt{r_1}(t,z) \\
& \wt{v}(t,z)=\sum_{j=1}^n   e^{-i(\mu_j +2\delta)t}      L_{-2\delta t} V_j(z+\alpha_j t) e^{\frac{1}{2}(\overline z \alpha_j - z \overline{\alpha_j})t} +\wt{r_2}(t,z) ,
\end{aligned}
\right.
\end{equation*}
where for all $s\geq 0$ and all $t \geq 0$
  \begin{equation} \label{conji}
  \|\<z\>^s (\partial^m_t \wt{r_1})(t) \|_{L^2}+  \|\<z\>^s(\partial^m_t \wt{r_2})(t) \|_{L^2} \leq C_{s,m} e^{-c_{s,m} t^2}.
  \end{equation}
  We refer to paragraph~\ref{para-disp} for a proof of \eqref{conji}.
\end{rem}

\subsubsection{A uniqueness result in $\mathcal{X}^{\kappa}_\mathcal{E}$} We are able to prove that the multi-soliton constructed in Theorem~\ref{thm-multi} is actually unique in the class $\mathcal{X}^{\kappa}_\mathcal{E}$, provided that $\kappa>0$ is large enough : 

\begin{thm}\label{thm-uni}
Let $n \geq 1$. For $1 \leq j \leq n$, let  $(K_j, a_j, b_j, \theta_j, \gamma_j ) \in \R^*_+ \times \R \times  \R \times \R \times \C$ and consider the parameters $(\lambda_j, \mu_j,\alpha_j) \in \R \times \R \times \C^*$  given by~\eqref{def-para1} and~\eqref{def-para2}. Assume that $\alpha_j \neq \alpha_{\ell}$ for $j \neq \ell$. Set
 \begin{equation} \label{dela}
\delta=  \frac{\dis  \max_{1 \leq j \leq n}K^2_j}{ \dis  \min_{1 \leq j \leq n}K^2_j} =\frac{\dis  \max_{1 \leq j \leq n}|\alpha_j|}{ \dis  \min_{1 \leq j \leq n}|\alpha_j|}.
  \end{equation}
There exists a universal constant $c_0>0$ such that if  $\kappa>c_0 \delta$ and if $(\wt{u},\wt{v})\in \mathcal{C}   \big(\R , \mathcal{X}^{\kappa}_\mathcal{E} \times \mathcal{X}^{\kappa}_\mathcal{E}\big)$ is a solution to equation~\eqref{eq2} of the form 
 \begin{equation} \label{multiU}
\left\{
\begin{aligned}
&\wt{u}(t,z)=\sum_{j=1}^n e^{-i\lambda_j t}U_j(z+\alpha_j t) e^{\frac{1}{2}(\overline z \alpha_j - z \overline{\alpha_j})t}+\wt{r_1}(t,z) \\
& \wt{v}(t,z)=\sum_{j=1}^n   e^{-i \mu_j t}V_j(z+\alpha_j t) e^{\frac{1}{2}(\overline z \alpha_j - z \overline{\alpha_j})t}  +\wt{r_2}(t,z) ,
\end{aligned}
\right.
\end{equation}
where the $(U_j,V_j)$ take the form \eqref{prog} and where  
  \begin{equation} \label{ass-exp}
 \big \|e^{\kappa|z|}  \wt{r_1}(t)  \big\|_{L^2}+  \big \|e^{\kappa|z|} \wt{r_2}(t)  \big\|_{L^2} \longrightarrow 0, \qquad t \longrightarrow +\infty,
  \end{equation}
then $(\wt{u},\wt{v})\equiv (u,v)$, where $(u,v)$ is given in Theorem~\ref{thm-multi}.
\end{thm}

In particular, if $\kappa>c_0 \delta$, the solutions constructed in Theorem~\ref{thm-multi} do not depend on $\kappa$. The assumption \eqref{ass-exp} is consistent with the result of Theorem~\ref{thm-multi}, but this assumption is quite strong. It would be interesting to relax it by asking only decay in $L^{2,s}_{\E}$ for some $s\geq 0$, but the situation would   more involved in this case. Actually, the assumption \eqref{ass-exp} implies an exponential decay in time of the error term and  as a consequence the interaction terms can quite easily be controlled.  \medskip

Contrarily to the Theorem~\ref{thm-multi}, in the previous result, one needs the assumption  $\alpha_j\neq 0$ for all  $1 \leq j \leq n$.
However, the result of Theorem~\ref{thm-uni} holds true for any traveling waves satisfying the weaker assumption~\eqref{gb}, but in this latter case,   the threshold is
 \begin{equation} \label{mod}
 \wt{\delta}=  \frac{\dis  \max_{1 \leq j \leq n}K^2_j}{ \dis  \min_{1 \leq j \leq n}|\alpha_j|},
 \end{equation}
where $K_j=\|U_j\|_{L^2}= \|V_j\|_{L^2}$. The modification \eqref{mod} comes from the fact that  one does no more necessarily have the relation \eqref{def-para2} for a general traveling wave, but only an inequality $\dis |\alpha_j| \leq \frac{ K_j^2}{2\sqrt{2} \pi}$ (see~\cite[Proposition~1.8]{Schw-Tho}). \medskip

Notice that the conditions \eqref{dela} and \eqref{mod} are consistent with the symmetries of the problem. In particular, the conditions  are invariant by scaling : if $(u,v)$ is a solution to \eqref{eq2},  then for all $A>0$, $(u_A,v_A)$ defined by $\big(u_A(t,z),v_A(t,z)\big)=\big(Au(A^2t,z),Av(A^2t,z)\big)$ is also a solution and under this transformation one has $(K,\alpha) \mapsto (AK, A^2\alpha)$.\medskip

The multi-soliton enjoys a   rigidity property. Consider a multi-soliton of the form \eqref{multiU} where the remainder terms satisfy \eqref{ass-exp} with $\kappa=0$. Then either   $(\wt{u},\wt{v})\equiv (u,v)$, where $(u,v)$ is given in Theorem~\ref{thm-multi} or there exist $C,c>0$ such that for all $t \in \R$
 $$\|\wt{r}_1(t)\|_{L^2(\C)}+ \|\wt{r}_2(t)\|_{L^2(\C)} \geq  C e^{-c |t|}, $$
 see Lemma~\ref{lemri}. In other words, there is only one multi-soliton which enjoys a Gaussian decay in time. A similar property holds true for logNLS \cite{Fer2}.

\subsubsection{Nonlinear superposition principle}  The next result shows that if one starts from a sum of traveling waves which all have the same speed but which are localised far away enough, then one has a good description of the dynamics of the solution to~\eqref{eq2} for long times, depending on the relative distance of the traveling waves.

 \begin{thm}\label{thm-super}
Let $(K, \theta)\in \R^*_+ \times \R$ and set $\dis \alpha = \frac{\sqrt{3}}{32\pi}K^2e^{-i\theta}$. Let $n \geq 1$ and for $1 \leq j \leq n$, let  $(a_j, b_j,  \gamma_j ) \in  \R \times   \R \times \C$ and consider the parameters $(\lambda_j, \mu_j) \in \R \times \R$  given by~\eqref{def-para1}. Assume that $\gamma_j \neq \gamma_{\ell}$ for $j \neq \ell$, and denote by 
$$\eps = \min_{j \neq \ell}{|\gamma_j-\gamma_\ell|}.$$
Consider the  solution $(u,v)\in \mathcal{C}^{\infty}  \big(\R , \mathcal{E} \times  \mathcal{E}\big)$ to equation~\eqref{eq2} such that
 \begin{equation*} 
u_0(z)=\sum_{j=1}^n  U_j(z),    \qquad  v_0(z)=\sum_{j=1}^n    V_j(z)    ,
\end{equation*}
where the $(U_j,V_j)$ take the form \eqref{prog}. Then 
 \begin{equation*} 
\left\{
\begin{aligned}
&u(t,z)=\sum_{j=1}^n e^{-i\lambda_j t}U_j(z+\alpha t) e^{\frac{1}{2}(\overline z \alpha - z \overline{\alpha})t}+r_1(t,z) \\
& v(t,z)=\sum_{j=1}^n   e^{-i \mu_j t}V_j(z+\alpha  t) e^{\frac{1}{2}(\overline z \alpha - z \overline{\alpha})t}  +r_2(t,z) ,
\end{aligned}
\right.
\end{equation*}
 and where the error terms satisfy :  there exist absolute constants $c,C>0$   such that for all $t \in \R$ 
  \begin{equation}\label{boundr}  
  \|    r_1(t)  \|_{L^2}+   \|  r_2(t)   \|_{L^2} \leq C K^2 \sqrt{|t|}e^{-\frac{ \eps^{-2}}4+c n^2K^2 |t|}.
  \end{equation}
\end{thm}

In particular for $|t|\leq  \eps^{-2}/(10c n^2K^2)$, then 
  \begin{equation*}  
  \|  r_1(t)   \|_{L^2}+   \|   r_2(t)   \|_{L^2} \leq C  e^{-\frac{ \eps^{-2}}8}.
  \end{equation*}

The proof of Theorem~\ref{thm-multi} is in the same spirit as the proof of Theorem~\ref{thm-multi} : in the present case, smallness is obtained thanks to the large distance between the waves ($\eps \ll 1$) instead of considering large times as in Theorem~\ref{thm-multi}. This result can be compared with \cite[Theorem 1.10]{Fer1} where a similar phenomenon occurs for the logNLS equation.\medskip

By a slight modification of our analysis, as in Theorem \ref{thm-multi}, one should also be able to obtain bounds  for $(\partial_t^mr_1, \partial_t^mr_2)$ and/or work in $\mathcal{X}^{\kappa}_\E$ spaces, but we do not write the details here.

 \subsection{Unbounded dynamics  for 2D linear harmonic oscillator}
 
 The result of Theorem~\ref{thm-multi} allows us to give new examples of unbounded trajectories to   the 2D linear harmonic oscillator
  \begin{equation}\label{harm} 
\left\{
\begin{aligned}
&i \partial_t \psi-H\psi +V(t,x,y)\psi=0, \qquad (t,x,y) \in \R \times \R^2,\\
&\psi(0,\cdot)=  \psi_0 \in L^2(\R^2).
\end{aligned}
\right.
\end{equation}
Recall the definition \eqref{def-sobo} of the Sobolev space $\HH^\sigma(\C)$. Our   result for the equation~\eqref{harm} reads as follows : 
 
 \begin{thm}\label{thm-lin}
Let $n \geq 1$. For $1 \leq j \leq n$, let  $(K_j, a_j, b_j, \theta_j, \gamma_j ) \in \R^*_+ \times \R \times  \R \times \R \times \C$ and consider the parameters $(\lambda_j, \mu_j,\alpha_j) \in \R \times \R \times \C^*$  given by~\eqref{def-para1} and~\eqref{def-para2}. Assume that $\alpha_j \neq \alpha_{\ell}$ for $j \neq \ell$. Then  there exists a potential $V \in \mathcal{C}^{\infty}(\R \times \R^2; \R)$ such that for all   $\sigma \geq 0$ and all  $k \in \N$
\begin{equation}\label{potendecay}
\lim_{t \to +\infty} \| \partial^k_t V(t)  \|_{\HH^\sigma(\C)} = 0,
\end{equation}
and there exists a solution  $\psi \in \mathcal{C}^{\infty}(\R \times \R^2; \C)$ to the equation \eqref{harm} of the form
\begin{equation*}
\psi(t)= \sum_{j=1}^n e^{-i\lambda_j \ln t}e^{-2i t}  L_{-2  t} R_{\alpha_j \ln t }U_j  +\eta(t),
\end{equation*}
where  $\|\eta(t)\|_{\HH^1(\C)} \longrightarrow 0$, when $t \longrightarrow +\infty$.
\end{thm}

In particular, for all $1\leq j \leq n$, 
$$   \|e^{-i\lambda_j \ln t}e^{-2i t}  L_{-2  t} R_{\alpha_j \ln t }U_j  \|_{\HH^1(\C)} =  \| R_{\alpha_j \ln t }U_j  \|_{\HH^1(\C)}\sim c_j \ln t , \quad  t \longrightarrow +\infty,  $$
for some $c_j>0$. The previous term has Gaussian decay and is concentrated near the point $x+iy  \sim -\alpha_j \ln t$. Therefore, $\psi$ is a sum of space-localised bubbles and
$$\|\psi(t)\|_{\HH^1(\C)} \sim \big(\sum_{j=1}^nc_j \big)  \ln t , \quad  t \longrightarrow +\infty.$$

 The result of Theorem~\ref{thm-lin} is a direct application of~\cite[Proposition 7.1]{Faou-Rapha} (see also~\cite[Theorem~1.1]{Faou-Rapha}), using  the solutions constructed in Theorem~\ref{thm-multi}.

\subsection{Analysis in the Bargmann-Fock space and notations} We end this section by recalling a few results and fixing some notations. The harmonic oscillator $H$ is defined by
$$
H = -4\partial_z \partial_{\ov z}+|z|^2=-(\partial^2_x+\partial^2_y)+(x^2+y^2).
$$
Denote by $(\phi_n)_{n \geq 0}$ the family of the special Hermite functions given by 
$$
\varphi_n(z) = \frac{1}{\sqrt{\pi n!}} z^n e^{-\frac{|z|^2}{2}}.
$$
The family $(\phi_n)_{n \geq 0}$ forms  a Hilbertian basis of $\mathcal{E}$ (see \cite[Proposition 2.1]{Zhu}),  and the $\phi_n$ are the eigenfunctions of $H$, namely 
$$H\phi_n=2(n+1)\phi_n, \quad n\geq 0.$$
For $\gamma \in \C$, we define 
\begin{equation*}
\phi_n^\gamma(z) = R_{- \overline \gamma} (\phi_n)(z) = \frac{1}{\sqrt{\pi n!}} (z-\overline \gamma)^n e^{-\frac{|z|^2}{2}-\frac{|\gamma|^2}{2} + \gamma z}.
\end{equation*} 

The   kernel of  $\Pi$, the orthogonal projection on $\mathcal{E}$, is explicitly given by
\begin{equation*}
K(z,\xi)=\sum_{n=0}^{+\infty}\phi_n(z)\ov{\phi_n}(\xi)=\frac{1}{\pi}e^{\ov{\xi}z}e^{-\vert \xi\vert^2/2}e^{-\vert z\vert^2/2}, \quad   (z,\xi)\in \C\times \C,
\end{equation*} 
and therefore we get the formula 
\begin{equation*} 
[\Pi u](z) = \frac{1}{\pi} e^{-\frac{|z|^2}{2}} \int_\mathbb{C} e^{\ov  w z - \frac{|w|^2}{2}} u(w) \,dL(w),
\end{equation*}
where $L$ stands for Lebesgue measure on $\C$. 

%With this  formula, we can compute   the following product rule (see \cite[Lemma 8.1]{GHT1})
%   \begin{equation}  \label{pi-phi}
%\Pi\big(\phi_{n_1}\ov{\phi_{n_2}}\phi_{n_3}\big) = \left\{
%\begin{aligned}
%& \frac{1 }{2\pi} \frac{(n_1+n_3)!}{2^{n_1+n_3}\sqrt{n_1 !n_2 !n_3 !n_4 !}} \phi_{n_4} &\quad \text{if} \quad n_4:= n_1+n_3-n_2\geq 0 \;\\
%& 0&\quad \text{if} \quad n_4:= n_1+n_3-n_2 <0 .
%\end{aligned}
%\right.
%\end{equation}
%
%
%
%\medskip
%Throughout the paper we use the classical notations $z=x+iy$ and 
%$$\partial_z=\frac12(\partial_x-i\partial_y), \qquad \partial_{\ov{z}}=\frac12(\partial_x+i\partial_y).$$ 

We define the enlarged lowest Landau level space as
\begin{equation*} 
\widetilde{\mathcal E}=\Big\{ u(z) = e^{-\frac{|z|^2}{2}} f(z)\,,\;f \; \mbox{entire\ holomorphic}\, \Big\}\cap \mathscr{S}'({\mathbb C})=\Big\{ u\in \mathscr{S}'(\C ), {\pa_{\ov z}}u+\frac	{z}{2}u=0\Big\} \ .
\end{equation*}
By Carlen~\cite{Carlen}, for all~$u \in \widetilde{\mathcal{E}}$  the following hypercontractivity estimates hold true
\begin{equation}  \label{hyp}
\mbox{if \;$1 \leq p \leq q \leq +\infty$,} \qquad\left( \frac{q}{2\pi} \right)^{1/q} \| u \|_{L^q(\C)} \leq \left( \frac{p}{2\pi} \right)^{1/p} \| u \|_{L^p(\C)}.
\end{equation}

In this paper $c,C>0$ denote universal constants the value of which may change from line to line.

\subsection{Plan of the paper} The rest of the article is organized as follows. In Section~\ref{Sect-2} we prove the well-posedness result for exponentially localised initial conditions. Section~\ref{Sect-2} is devoted to technical results, while the next ones contain the proofs of the main theorems.

\section{Well-posedness and stability results}\label{Sect-2}

\subsection{Continuity results for the projector $\Pi$} The next result shows that $\Pi$ is continuous in~$\mathcal{X}^{\kappa}_\mathcal{E}$ spaces.
\begin{lem} Let $s \geq 0$ and $1 \leq p \leq +\infty$, then for all $F \in \mathscr{S}'(\C)$, 
\begin{equation}\label{conti}
\| \<z\>^s  \Pi(F) \|_{L^p} \leq C \| \<z\>^s  F\|_{L^p},
\end{equation} 
and for all $\kappa\geq 0$
\begin{equation}\label{conti2}
\|e^{\kappa|z|}  \Pi(F) \|_{L^p} \leq C_\kappa \|e^{\kappa|z|}   F\|_{L^p}.
\end{equation} 
\end{lem}

\begin{proof}
The bound \eqref{conti} is proved in~\cite[Proposition 3.1]{GGT}. Let us show \eqref{conti2}. For $F \in \mathscr{S}'(\C)$  we  have 
$$\Pi (F)(z)=\frac{e^{-\frac{|z|^2}2}}\pi  \int_\mathbb{C} e^{\ov  w z - \frac{|w|^2}{2}} F(w) \,dL(w),$$
and therefore, using that   $| e^{ - \frac{|z|^2}{2}+\ov  w z - \frac{|w|^2}{2}} |=   e^{-\frac{|z-w|^2}2}$ and $e^{\kappa|z|} \leq e^{\kappa|z-w|} e^{\kappa|w|}$ we get
\begin{equation*}
e^{\kappa|z|} |\Pi (F)(z)| \leq \frac{1}{\pi}   \int_\mathbb{C} e^{\kappa|z-w|-\frac{|z-w|^2}2} |e^{\kappa|w|} F(w)| \,dL(w)=\big(\psi \star (e^{\kappa|\cdot|}|F|)\big)(z),
  \end{equation*}
where $\psi(z)= \frac{1}\pi e^{\kappa|z|-|z|^2/2} \in L^1(\C)$.  Therefore  by the  Young inequality 
$$\| e^{\kappa|z|}  \Pi(F)  \|_{L^p(\C)} \leq \|\psi \|_{L^1(\C)} \| e^{\kappa |z|}  F \|_{L^p(\C)}  \leq C e^{\kappa^2/2} \| e^{\kappa |z|}  F \|_{L^p(\C)} $$
which is \eqref{conti2}.
 \end{proof}

\subsection{Proof of Proposition \ref{propExp}} 

The proof of Proposition~\ref{propExp} follows the lines of the proof of~\cite[Theorem~1.1]{Schw-Tho}. We also refer to~\cite[Section~3]{GGT} for other  well-posedness results for the LLL equation. \medskip

$\bullet$ {\it Proof of the global existence in $\mathcal{X}^{\kappa}_\mathcal{E}$.} By~\eqref{conti2} and \eqref{hyp} we obtain
\begin{eqnarray}\label{est-exp}
\|  e^{\kappa|z|} \Pi \big( a b c\big) \|_{L^2}  &\leq&C_{\kappa} \|e^{\kappa|z|}  a\|_{L^2} \| b\|_{L^\infty}\| c\|_{L^\infty}\nonumber \\
&\leq&C_{\kappa} \|e^{\kappa|z|}  a\|_{L^2} \| b\|_{L^2}\| c\|_{L^2}.
\end{eqnarray}
The estimate \eqref{est-exp} allows for the  construction of a local in time solution with a fixed point argument, and the globalisation is obtained  using that  the time of existence  only depends on the~$L^2$ norm of the solution.
\medskip

$\bullet$ {\it Proof of \eqref{mexp}.} Let $(u_0,v_0) \in  \mathcal{X}^{\kappa}_\mathcal{E} \times \mathcal{X}^{\kappa}_\mathcal{E}$ and consider $(u,v)\in \mathcal{C}^{\infty}  \big(\R , \mathcal{X}^{\kappa}_\mathcal{E} \times \mathcal{X}^{\kappa}_\mathcal{E}\big)$ the corresponding solution  to equation~\eqref{eq2}.  We  compute 
     \begin{eqnarray*} 
  \frac{d}{dt} \int_\C e^{2\kappa |z|}|u|^2dL &=& 2 \Re \int_\C e^{2\kappa |z|}\ov{u} \partial_t udL \\
   &=& 2 \Im \int_\C e^{2\kappa|z|}\ov{u} \, \Pi (|v|^2 u) dL \\
      &\leq & 2   \|e^{\kappa|z|} u\|_{L^2}     \big\|e^{\kappa|z|} \Pi (|v|^2 u)\big\|_{L^2}.  
  \end{eqnarray*}
Next, by \eqref{conti2} we get 
   \begin{eqnarray} 
  \frac{d}{dt} \int_\C e^{2\kappa |z|}|u|^2dL  
       &\leq & C   \|e^{\kappa|z|} u\|_{L^2}     \big\|e^{\kappa|z|} |v|^2 u\big\|_{L^2}  \nonumber\\
          &\leq & C   \|e^{\kappa|z|} u\|^2_{L^2}     \| v\|^2_{L^\infty} \nonumber \\
                  &\leq & C   \|e^{\kappa|z|} u\|^2_{L^2}     \| v\|^2_{L^2}   \nonumber \\
                                    &\leq & C   \|e^{\kappa|z|} u\|^2_{L^2}     \| v_0\|^2_{L^2}\label{control22}   , 
  \end{eqnarray}
where  we used the Carlen inequality \eqref{hyp} and the conservation of the $L^2$-norm. We conclude by integration.\medskip

$\bullet$ {\it Proof of \eqref{mexp2}.} The proof of this inequality is in the same spirit as the previous one. We have
     \begin{eqnarray*} 
i \partial_t(u-\wt{u}) &=&  \Pi\big(|v|^2u-|\wt{v}|^2\wt{u}\big) \\
   &=&  \Pi\big((u-\wt{u})|v|^2+  (v-\wt{v})\wt{u}\ov{v}+\wt{v}(\ov{v}-\ov{\wt{v}})\wt{u} \wt{v} \big).  
  \end{eqnarray*}
Then, with the same arguments as in~\eqref{control22}, we get 
  \begin{multline*} 
  \frac{d}{dt} \int_\C e^{2\kappa |z|}|u-\wt{u}|^2dL  
       \leq  \\
 \begin{aligned}   
 &\leq   C_\kappa   \|e^{\kappa|z|} (u-\wt{u})\|_{L^2}  \big(\|e^{\kappa|z|} (u-\wt{u})\|_{L^2}+  \|e^{\kappa|z|} (v-\wt{v})\|_{L^2}\big)  \big(    \| u\|^2_{L^\infty}  +   \| \wt{u}\|^2_{L^\infty}+   \| v\|^2_{L^\infty}  +   \| \wt{v}\|^2_{L^\infty}\big)\\
&\leq   C_\kappa   \|e^{\kappa|z|} (u-\wt{u})\|_{L^2}  \big(\|e^{\kappa|z|} (u-\wt{u})\|_{L^2}+  \|e^{\kappa|z|} (v-\wt{v})\|_{L^2}\big)  \big(    \| u_0\|^2_{L^2}  +   \| \wt{u}_0\|^2_{L^2}+   \| v_0\|^2_{L^2}  +   \| \wt{v}_0\|^2_{L^2}\big).
 \end{aligned}  
  \end{multline*}    
  Therefore, setting $\theta=   \|e^{\kappa|z|} (u-\wt{u})\|^2_{L^2}+  \|e^{\kappa|z|} (v-\wt{v})\|^2_{L^2}$, we get the bound $\theta'(t) \leq C_\kappa \theta(t)$ and we deduce \eqref{mexp2} by integration.      
  
  \subsection{A rigidity result}
  A direct consequence   of \eqref{mexp2} (with $\kappa=0$) is  the following rigidity result for the system~\eqref{eq2} in $L^2(\C)$ :

 \begin{lem}\label{lemri}  Let $(u, v) \in \mathcal{C}^{\infty} (\R, \E \times \E)$ and $(\wt{u}, \wt{v}) \in \mathcal{C}^{\infty} (\R, \E \times \E)$ be  solutions to \eqref{eq2}. Then   there exists a universal constant $c>0$ such that for all
 $t \in \R$
 \begin{multline*} 
\|u(t)-\wt{u}(t)\|^2_{L^2(\C)}+ \|v(t)-\wt{v}(t)\|^2_{L^2(\C)} \geq \\
\geq \big(\|u_0-\wt{u}_0\|^2_{L^2(\C)}+ \|v_0-\wt{v}_0\|^2_{L^2(\C)} \big) e^{-c   (\|  u_0\|^2_{L^2}+ \|  \wt{u}_0\|^2_{L^2}+\|  v_0\|^2_{L^2}+ \|  \wt{v}_0\|^2_{L^2}) |t|}.
 \end{multline*}
  \end{lem}  
 
 In other words, either  $(u, v) = (\wt{u}, \wt{v})$ or there exist $c,C>0$ such that 
 $$\|u(t)-\wt{u}(t)\|_{L^2(\C)}+ \|v(t)-\wt{v}(t)\|_{L^2(\C)}
\geq Ce^{-c  |t|}.$$
Notice that a similar property holds true for the Schr\"odinger equation with logarithmic nonlinearity, see \cite[Lemma 6.2]{Fer2}.

    %%%%%%%%%%%%%%%%%%%%%%%%%%%%%%%%%%%%%%%%%%%%%%%%%%%%%%%%%%%

\section{Preliminary results}

  \subsection{Interactions of traveling waves} The next result shows that the interactions between the traveling waves have a Gaussian decay with respect to their relative distance.
    \begin{lem}\label{lem-exp}
Assume that $U_1,U_2 \in \E$ satisfy for all $z \in \C$
   \begin{equation}\label{B0}
|U_1(z)| \leq C_0 e^{-c_0{|z|^2}}, \quad |U_2(z)| \leq C_0 e^{-c_0{|z|^2}},
      \end{equation}
for some $c_0,C_0>0$.  Then 
\begin{enumerate}[$(i)$]
\item there exists   $C>0$ such that for all  $\alpha_1, \alpha_2 \in \C$
   \begin{equation}\label{B1}
   \|  (R_{\alpha_1} U_1)\, (R_{\alpha_2} U_2)  \|_{L^\infty}  \leq Ce^{-\frac{c_0}2 |\alpha_1-\alpha_2|^2} \,;
      \end{equation}
\item     for all $s\geq 0$ and all $c_1<c_0$, there exists   $C>0$ such that for all  $\alpha_1, \alpha_2 \in \C$
       \begin{equation}\label{B2}
    \| \<z\>^s (R_{\alpha_1} U_1)\, (R_{\alpha_2} U_2)  \|_{L^2}  \leq C\min\big(\<\alpha_1\>^s, \<\alpha_2\>^s\big)e^{-\frac{c_1}2|\alpha_1-\alpha_2|^2} \,;
          \end{equation}
 \item       for all $\kappa >0$ and all $c_1<c_0$    there exists   $C>0$ such that for all  $\alpha_1, \alpha_2 \in \C$
       \begin{equation*} 
    \| e^{\kappa |z|}(R_{\alpha_1} U_1)\, (R_{\alpha_2} U_2)  \|_{L^2}  \leq C\min\big(e^{\kappa |\alpha_1|}, e^{\kappa |\alpha_2|}\big)e^{-\frac{c_1}2|\alpha_1-\alpha_2|^2} ;
          \end{equation*}
           \item            there exists   $C>0$ such that for all  $L>0$ and $\alpha_1 \in \C$
       \begin{equation}\label{B4}
    \| e^{-L|z|}R_{\alpha_1} U_1 \|_{L^{\infty}}  \leq C e^{-L |\alpha_1| /2}+C e^{-c_0 |\alpha_1|^2/4 }.
          \end{equation}
          \end{enumerate}
    \end{lem}
    
      \begin{proof}
$(i)$      First, we observe that we have the relation
      $$|z+\alpha_1|^2+ |z+\alpha_2|^2= 2\big|z+\frac{\alpha_1+\alpha_2}2\big|^2 +\frac12 |\alpha_1-\alpha_2|^2.$$
    Therefore   by \eqref{B0}
      \begin{eqnarray} \label{az}
\big|  (R_{\alpha_1} U_1)(z) (R_{\alpha_2} U_2)(z) \big| &\leq &C e^{-c_0{|z+\alpha_1|^2}-c_0{|z+\alpha_2|^2}   }  \nonumber\\
&=  &C e^{-c_0\frac{|\alpha_1-\alpha_2|^2}2     }e^{- 2c_0|z+\frac{\alpha_1+\alpha_2}2|^2}  , 
   \end{eqnarray}  
hence the estimate \eqref{B1}.  

$(ii)$ Assume for instance $|\alpha_2| \leq |\alpha_1|$.   In order to prove \eqref{B2} we write 
      \begin{eqnarray*} 
    \| \<z\>^s (R_{\alpha_1} U_1)\, (R_{\alpha_2} U_2)  \|_{L^2}  &=& \| \<z-\alpha_2\>^s (R_{\alpha_1-\alpha_2} U_1)\,  U_2  \|_{L^2} \\
    &\leq & C\<\alpha_2\>^s   \| (R_{\alpha_1-\alpha_2} U_1)\,  U_2  \|_{L^2} + C \| (R_{\alpha_1-\alpha_2}  U_1)\,   \<z\>^sU_2  \|_{L^2} .
          \end{eqnarray*} 
   Observe that    $  \<z\>^s|U_2(z)| \leq C_0 e^{-c_1{|z|^2}}$  for all $c_1<c_0$  and  therefore by~\eqref{az} we obtain 
                 \begin{equation*}
              \| \<z\>^s (R_{\alpha_1} U_1)\, (R_{\alpha_2} U_2)  \|_{L^2}  \leq C \<\alpha_2\>^s e^{-\frac{c_1}2|\alpha_1-\alpha_2|^2} ,
                     \end{equation*}   
  which was to prove.
  
  $(iii)$ Similarly, for all $c_1<c_0$, we have
    \begin{eqnarray*} 
    \| e^{\kappa |z|} (R_{\alpha_1} U_1)\, (R_{\alpha_2} U_2)  \|_{L^2}  &=& \| e^{\kappa |z-\alpha_2|} (R_{\alpha_1-\alpha_2} U_1)\,  U_2  \|_{L^2} \\
    &\leq & e^{\kappa |\alpha_2|}    \| (R_{\alpha_1-\alpha_2} U_1)\, (e^{\kappa |z|}  U_2 ) \|_{L^2}    \\
 &\leq &     Ce^{\kappa |\alpha_2|}  e^{-\frac{c_1}2|\alpha_1-\alpha_2|^2}  ,
      \end{eqnarray*} 
      hence the result.
      
        $(iv)$ By hypothesis \eqref{B0}
             \begin{equation}\label{BB}
    | e^{-L|z|}R_{\alpha_1} U_1 |  \leq  C e^{-L|z|-c_0 |z+\alpha_1|^2 }.
          \end{equation}
      Then observe that     
           \begin{equation*}  
L|z|  + c_0|z+\alpha_1|^2 \geq  \left\{
\begin{aligned}
\; L|\alpha_1|/2 &\quad \text{if} \quad  |z|\geq |\alpha_1|/2 \;\\
\;c_0|\alpha_1|^2/4  &\quad \text{if} \quad  |z|\leq |\alpha_1|/2 ,
\end{aligned}
\right.
\end{equation*}
which implies the result by \eqref{BB}.
  \end{proof}
    
    \subsection{Stability of  traveling waves under time derivation} 
   The next lemma shows that the Gaussian decay of the traveling waves is stable under time derivation. 
    
      \begin{lem}\label{lem-X}
Let $c_0\leq 1/2$. Assume that    $U\in \E$ satisfies a Gaussian bound $|U(z)| \le C_0 e^{-c_0 |z|^2}$ and assume that  $T \in \mathcal{C}\big(\R; \E\big)$ takes the form 
$$T(t,z)=e^{-it \lambda}(R_{\alpha t}U)(z),\qquad \forall \, (t, z)\in \R \times \C.$$ 
Then there exists $\wt{U} \in \E$ such that  for all $c_1<\frac{c_0}{1+2c_0}$, there exists $C_1>0$
$${|\wt{U}(z)| \le C_1 e^{-c_1 |z|^2}}$$
and  
$$(\partial_t T)(t,z)=e^{-it \lambda}(R_{\alpha t}\wt{U})(z),\qquad \forall \, (t, z)\in \R \times \C.$$ 
    \end{lem}

\begin{proof} We set $U(z)=f(z) e^{-\frac{|z|^2}2}$, thus
$$
T(t,z)= e^{-i\lambda t}U(z+\alpha t) e^{\frac{1}{2}(\overline z \alpha  - z \overline{\alpha })t}=e^{-i\lambda t}f(z+\alpha t) e^{ - z \overline{\alpha }t-\frac{|z|^2}2-\frac{|\alpha |^2t^2}2}.
$$
A direct computation gives $ (\partial_t T)(t,z)=e^{-i\lambda t}{\wt U}(z+\alpha t) e^{\frac{1}{2}(\overline z \alpha  - z \overline{\alpha })t}$ where 
$${\wt U}(z)=\big(-i \lambda f(z) +\alpha \partial_z f(z) -\ov{\alpha} zf(z) \big)e^{-\frac{|z|^2}2}.$$
Let $c_1<\frac{c_0}{1+2c_0}$. It remains to check  that $\wt{U}$ satisfies a Gaussian bound with constant $c_1$. Since $c_1<c_0$, it is clear that    $|zU(z)| \le C e^{-c_1 |z|^2}$, and let us prove  that $|\partial_z f(z) | \e^{-\frac{|z|^2}2} \leq C \e^{-c_1{|z|^2}}$. Observe that 
$$(\partial_z f)(z)  \e^{-\frac{|z|^2}2}=\big( \partial_z-\frac{\ov{z}}2  \big)U(z)+\ov{z}U(z),$$
hence it is enough to show that $\big|\big( \partial_z-\frac{\ov{z}}2  \big)U(z)\big| \e^{-\frac{|z|^2}2} \leq C \e^{-c_1{|z|^2}}$. 
Writing 
$$U(z)=\Pi(U)(z)=\frac{e^{-\frac{|z|^2}2}}\pi  \int_\mathbb{C} e^{\ov  w z - \frac{|w|^2}{2}} U(w) \,dL(w),$$
we obtain
$$\big(\partial_z -\frac{\ov z}2\big)U(z)=\frac{e^{-\frac{|z|^2}2}}\pi  \int_\mathbb{C} (\ov w-\ov z)e^{\ov  w z - \frac{|w|^2}{2}} U(w) \,dL(w).$$
This in turn implies
\begin{equation}\label{impl}
  \Big|\big(\partial_z -\frac{\ov z}2\big)U(z)   \Big| \leq  C  \big\| e^{c_0  |z|^2 }U\big\|_{L^{\infty}} \int_\mathbb{C} |z-w| e^{-\frac{|z-w|^2}2}   e^{-c_0{|w|^2}}  \,dL(w).
    \end{equation}
  
  Now let us show that there exists $\eps >0$ such that for all $w, z \in \C$
  \begin{equation}\label{poly}
  \frac12|z-w|^2+c_0 |w|^2 \geq \eps |z-w|^2+c_1 |z|^2.
      \end{equation}
  We can assume that $w, z \in \R$ and by homogeneity we can reduce to the case $w=1$. Define the polynomial 
    \begin{eqnarray*}
P(z) &=&  \frac12 (z-1)^2+c_0  - \eps (z-1)^2-c_1 z^2 \\
   &= & (\frac12 -c_1-\eps)z^2-(1-2\eps) z +  (\frac12 +c_0-\eps).
  \end{eqnarray*}
  The discriminant of $P$ is $\Delta= 2(c_1-c_0)(1-2\eps)+4c_0 c_1$ which is negative for $\eps >0$ small enough, since $c_1<\frac{c_0}{1+2c_0}$. As a consequence, $P\geq 0$ for $\eps>0$ small enough, which implies \eqref{poly}. From~\eqref{impl} and \eqref{poly} we deduce that 
  \begin{eqnarray*}
  \Big|\big(\partial_z -\frac{\ov z}2\big)U(z)   \Big| 
  &\leq & C  \big\| e^{c_0  |z|^2 }U\big\|_{L^{\infty}} e^{-c_1 |z|^2}    \int_\mathbb{C} |z-w| e^{-\eps {|z-w|^2}}   \,dL(w) \\
   &\leq & C e^{-c_1 |z|^2} ,
  \end{eqnarray*}
which concludes the proof.
    \end{proof}

\section{Proof of Theorem~\ref{thm-multi}}

 We recall the system
 \begin{equation}\label{eq1} 
\left\{
\begin{aligned}
&i\partial_{t}u= \Pi (|v|^2 u), \quad   (t,z)\in \R\times \C,\\
&i\partial_{t}v=  -\Pi (|u|^2 v),\\
&u(0,\cdot)=  u_0 \in\E,\; v(0,\cdot)=  v_0 \in\E.
\end{aligned}
\right.
\end{equation}
Assume that for each $1\leq j \leq n$, $(U_j,V_j)\in  \E \times \E$ is a traveling wave solution to \eqref{eq1} in the sense~\eqref{trav}, and that there exist $c_0,C>0$ such that 
\begin{equation}\label{gb1}
|U_j(z)|+|V_j(z)| \leq C e^{-c_0 |z|^2}.
\end{equation}
One can have $\alpha_j=0$ for at most one $1\leq j\leq n$, and in this case, \eqref{gb1} is automatically satisfied for any $c_0<1/2$, by Theorem~\ref{thmDec}.  We denote by $K_j=\|U_j\|_{L^2}= \|V_j\|_{L^2}$ (notice that one always has   $\|U_j\|_{L^2}= \|V_j\|_{L^2}$  for a traveling wave, see~\cite[Proposition~1.8]{Schw-Tho}), and we set
\begin{equation} \label{travj}
\begin{aligned} 
& X_j(t,z)= e^{-i\lambda_j t}U_j(z+\alpha_j t) e^{\frac{1}{2}(\overline z \alpha_j - z \overline{\alpha_j})t}\\[4pt]
&   Y_j(t,z)= e^{-i\mu_j t}V_j(z+\alpha_j t) e^{\frac{1}{2}(\overline z \alpha_j - z \overline{\alpha_j})t}  \;,
\end{aligned}
\end{equation}
 and  
 \begin{equation}\label{decomp}
 u=\sum_{j=1}^n X_j+r_1:=X+r_1,\qquad v=\sum_{j=1}^n Y_j+r_2:=Y+r_2,
 \end{equation}
 a solution of \eqref{eq1}.\medskip
 
 For $M>0$, let $(u^M, v^M) \in \mathcal{C}^{\infty}(\R , \E \times \E)$ be the solution to \eqref{eq1} such that $\big(r_1^M(M), r_2^M(M)\big)=(0,0)$. By Proposition~\ref{propExp}, we also have $(u^M, v^M) \in \mathcal{C}^{\infty}(\R ,  \mathcal{X}^{\kappa}_\mathcal{E} \times \mathcal{X}^{\kappa}_\mathcal{E})$ for all $\kappa\geq 0$, and hence
 $$(r_1^M, r_2^M) \in \mathcal{C}^{\infty}(\R ,  \mathcal{X}^{\kappa}_\mathcal{E} \times \mathcal{X}^{\kappa}_\mathcal{E}).$$
 
\subsection{The nonlinear analysis}
The next result shows that the remainder term has an explicit  Gaussian decay, with uniform constants with respect to $M>0$.

  \begin{lem}\label{lem-borne}
  Let $\kappa\geq 0$ and $c_1<\frac{c_0}2$. There exists a constant $C>0$ such that for all  $M>0$  and all $0<t \leq M$
    $$\big\|e^{\kappa |z|} r^M_1(t)\big\|_{L^2}+\big\|e^{\kappa |z|} r^M_2(t)\big\|_{L^2}  \leq Ce^{-{c_1} \alpha^2_{\sharp}t^2}.$$
    \end{lem}

  \begin{proof}
 Fix $M>0$. In the sequel, we write $r_1=r^M_1$, $r_2=r^M_2$  and we denote  by 
 $$\eta(t):=\|e^{\kappa |z|} r_1(t)\|^2_{L^2}+\|e^{\kappa |z|} r_2(t)\|^2_{L^2}.$$ We stress that all the constants $C,c_0,c_1,c_2>0$ below will be independent of $M>0$. \medskip

 \underline{Step 1 : A first $L^2$-bound.}  To begin with, let us prove that 
  \begin{equation} \label{bornER}
   \|r^M_1(t)\|_{L^2}  \leq 2 \sum_{j=1}^n  K_j \qquad  \|r^M_2(t)\|_{L^2}  \leq 2 \sum_{j=1}^n  K_j.
  \end{equation}
 By the conservation of the $L^2$ norm for $(u^M,v^M)$ and the triangle inequality, for all $t\in \R$
\begin{eqnarray*}
 \|r^M_1(t)\|_{L^2} &\leq& \|u^M(t)\|_{L^2}+ \sum_{j=1}^n \|X_j(t)\|_{L^2} \nonumber  \\
 &\leq& \|u^M(M)\|_{L^2}+ \sum_{j=1}^n \|X_j(t)\|_{L^2} \nonumber  \\
 &\leq& \sum_{j=1}^n \big(\|X_j(M)\|_{L^2}+ \|X_j(t)\|_{L^2}\big)\nonumber  \\
 &\leq &2 \sum_{j=1}^n  K_j,
  \end{eqnarray*}
 uniformly  with respect to $t\in \R$ and $M>0$.\medskip

\underline{Step 2 : A differential inequality.} In this paragraph we show that, for all $c_1<\frac{c_0}2$,  there exists a   constant $C>0$ independent of $M>0$ such that for all $0<t \leq M$
   \begin{equation}\label{equadiff}
\big|   \frac{d}{dt} \eta(t)\big|  \leq C \eta(t)+ C    e^{-2{c_1} \alpha^2_{\sharp}t^2}.
  \end{equation}
By \eqref{decomp}, the relation $i\partial_{t}u= \Pi (|v|^2 u)$ reads
   \begin{equation}\label{eqer}
 i\partial_{t}r_1= \Pi \big(|Y+r_2|^2 (X+r_1)\big)- i\partial_{t}X:= \sum_{j=0}^3q_j,
  \end{equation}
where
\begin{equation}  \label{exp-q}
\begin{aligned}
q_0&:=  \Pi \big(|Y|^2 X\big)- i\partial_{t}X \\
q_1&:=  \Pi \big(|Y|^2 r_1\big) + \Pi \big(\ov{Y}X r_2\big) + \Pi \big({Y}X \ov{r_2}\big)\\
q_2&:=  \Pi \big( X|r_2|^2\big)+\Pi \big( \ov{Y}r_1 r_2\big)+\Pi \big( {Y}r_1 \ov{r_2}\big)\\
q_3&:= \Pi \big(|r_2|^2 r_1\big).
\end{aligned}
\end{equation}

  Next we compute 
     \begin{eqnarray} 
\Big|  \frac{d}{dt} \int_\C e^{2\kappa|z|}|r_1|^2dL \Big| &=&\Big| 2 \Re \int_\C e^{2\kappa  |z|}\ov{r_1} \partial_t r_1dL \Big|\nonumber \\
   &=& \Big|2 \Im \int_\C e^{2\kappa  |z|}\ov{r_1}  \big(\sum_{j=0}^3q_j\big) dL \Big|\nonumber\\
      &\leq & C   \|e^{\kappa  |z|} r_1\|_{L^2}   \sum_{j=0}^3   \|e^{\kappa  |z|} q_j\|_{L^2}  , \label{sim}
  \end{eqnarray}
and we now  have to estimate each term $\|e^{\kappa  |z|}q_j\|_{L^2} $. 
\medskip

$\bullet$ Control of $\|e^{\kappa |z|}q_0\|_{L^2} $.  For all $1 \leq j \leq n$, $ i\partial_{t}X_j=  \Pi \big(|Y_j|^2 X_j\big) $, thus
$$q_0=\sum_{\substack{1 \leq j, k, \ell  \leq n \\ (j,k, \ell) \neq (j,j,j) }} \Pi \big(Y_j \ov{Y_k}X_{\ell} \big).   $$
Assume for instance that $j\neq k$, so that $\alpha_j \neq \alpha_k$. Then by \eqref{conti2} 
   \begin{eqnarray*} 
\| e^{\kappa  |z|}\Pi \big(Y_j \ov{Y_k}X_{\ell} \big) \|_{L^2}&\leq &C\| e^{\kappa |z|} Y_j {Y_k}X_{\ell}  \|_{L^2} \\
&= &C \| e^{\kappa |z|} (R_{\alpha_j t}V_j ) (R_{\alpha_k t}V_k)  (R_{\alpha_\ell t}U_\ell)  \|_{L^2}\\
&= &C \| e^{\kappa  |z-\alpha_\ell t|}  (R_{(\alpha_j-\alpha_\ell )t}V_j ) (R_{\alpha_k-\alpha_\ell )t}V_k)  U_\ell  \|_{L^2} \\
&\leq &C e^{\kappa |\alpha_\ell |t}  \|(R_{(\alpha_j-\alpha_\ell )t}V_j ) (R_{\alpha_k-\alpha_\ell )t}V_k)\|_{L^{\infty}}  \|e^{\kappa |z |} U_\ell  \|_{L^2} .
  \end{eqnarray*}
  Next by~\eqref{B1}, for all $c_1<\frac{c_0}2$ we have
     \begin{eqnarray*} 
\| e^{\kappa  |z|}\Pi \big(Y_j \ov{Y_k}X_{\ell} \big) \|_{L^2}
&\leq & C e^{\kappa |\alpha_\ell |t-\frac{c_0}2|\alpha_j-\alpha_k|^2t^2}\| e^{\kappa |z |} U_\ell  \|_{L^2} \\
&\leq & C   e^{-{c_1}|\alpha_j-\alpha_k|^2t^2} .
  \end{eqnarray*}
The other terms are treated similarly.  As a consequence, for all $c_1<\frac{c_0}2$  
   \begin{equation}\label{control-q0}
   \|e^{\kappa |z|}q_0\|_{L^2}  \leq C    e^{-{c_1} \alpha^2_{\sharp}t^2}.
  \end{equation}~
  
$\bullet$ Control of $\|e^{\kappa |z|}q_1\|_{L^2} $. From~\eqref{conti2} we have
   \begin{eqnarray*} 
\| e^{\kappa |z|} q_1 \|_{L^2}&\leq &C\| e^{\kappa |z|}  |Y|^2 r_1     \|_{L^2}    +C\| e^{\kappa |z|}\ov{Y}X r_2 \|_{L^2}   \\
&\leq &    C\|  Y     \|^2_{L^\infty}   \| e^{\kappa |z|}   r_1     \|_{L^2}    + C \|  X     \|_{L^\infty}  \|  Y     \|_{L^\infty}  \| e^{\kappa |z|}   r_2     \|_{L^2}  \\
&\leq &   C\big(  \| e^{\kappa |z|}   r_1     \|_{L^2}    +    \| e^{\kappa |z|}   r_2     \|_{L^2} \big) \\
&\leq &C \eta^{1/2} .
  \end{eqnarray*}~

$\bullet$ Control of $\|e^{\kappa |z|}q_2\|_{L^2} $. The estimation of this contribution is in the same spirit as the previous one. Firstly,
   \begin{equation} \label{prev}
\| e^{\kappa |z|} q_2 \|_{L^2} \leq  C  \big( \|  Y     \|_{L^\infty}  \|  r_1     \|_{L^\infty} + \|  X     \|_{L^\infty}  \|  r_2     \|_{L^\infty}  \big)  \| e^{\kappa |z|}   r_2     \|_{L^2}.
  \end{equation}~

Then by the Carlen estimate \eqref{hyp} and by \eqref{bornER}, 
   \begin{equation*} 
    \|  r_1     \|_{L^\infty}  \leq C \|  r_1     \|_{L^2}  \leq C,
      \end{equation*}
 and $ \|  r_2        \|_{L^2}  \leq C$ as well. Thus from \eqref{prev} we deduce 
   \begin{equation*} 
\| e^{\kappa |z|} q_2 \|_{L^2} \leq   C \eta^{1/2} .
  \end{equation*}

   $\bullet$ Control of $\|e^{\kappa |z|}q_3\|_{L^2} $. Similarly,
   \begin{equation*} 
\| e^{\kappa |z|} q_3 \|_{L^2} \leq  C \|  r_2     \|^2_{L^\infty}  \| e^{\kappa |z|}   r_1     \|_{L^2}  \leq C  \eta^{1/2}  .
  \end{equation*}~

Putting all the previous estimates together, from~\eqref{sim} we obtain 
   \begin{equation*} 
\Big|   \frac{d}{dt} \|e^{\kappa |z|}r_1(t)\|^2_{L^2}\Big|  \leq C\eta^{1/2}(t) \big(     \eta^{1/2}(t)+ Ce^{-{c_1} \alpha^2_{\sharp}t^2}\big)  \leq  C \eta(t)+ Ce^{-2{c_1} \alpha^2_{\sharp}t^2}.
  \end{equation*}
The estimate for $\dis \Big|   \frac{d}{dt} \|e^{\kappa |z|}r_2(t)\|^2_{L^2}\Big| $ is similar, hence we get \eqref{equadiff}.
  
  \medskip
\underline{Step 3 : Backward  Gr\"onwall.}  Now,  by integrating   \eqref{equadiff} on $[t, M]$ we get  that for all $0<t<M$ 
  $$\eta(t) \leq  C \int_{t}^{+\infty} e^{-2{c_1} \alpha^2_{\sharp}\sigma^2}d \sigma + C \int_{t}^M \eta(\sigma) d \sigma\leq   C   e^{-2{c_1} \alpha^2_{\sharp}t^2} +C \int_{t}^M \eta(\sigma) d \sigma. $$
   By the backward Gr\"onwall inequality (Lemma~\ref{lem-Gro}) this implies that for all $0<t<M$ 
    \begin{eqnarray*}
    \eta(t) &\leq &    C    e^{-2{c_1} \alpha^2_{\sharp}t^2}    +C \int_{t}^M  e^{-2{c_1} \alpha^2_{\sharp}\sigma^2}  \exp  \big(\int_t^\sigma C d\tau\big)        d \sigma \\
    &\leq &    C   e^{-2{c_1} \alpha^2_{\sharp}t^2}   +C \int_{t}^M  e^{-2{c_1} \alpha^2_{\sharp}\sigma^2+C \sigma}       d \sigma\\
        &\leq &    C   e^{-2{c_2} \alpha^2_{\sharp}t^2}  ,
      \end{eqnarray*}
  for any $c_2 <c_1$, and where the previous constant $C>0$ does not depend on $M>0$, which was the claim.
    \end{proof}

    We now prove that    for all $T>0$, the sequence $\big( r^M_1 , r^M_2 \big)_{M \geq 0}$ is a Cauchy sequence in the space $\mathcal{C} \big([0,T];   \mathcal{X}^{\kappa}_\mathcal{E} \times \mathcal{X}^{\kappa}_\mathcal{E})$ :
    
      \begin{lem}\label{lem-cauchy}
Let $\kappa\geq 0$. For all $c_1<\frac{c_0}2$,  there exists a constant $C>0$ such that for all  $0<N<M$ and  all $0<t \leq N$
  $$\big\|e^{\kappa |z|}(r^M_1-r^N_1)(t)\big\|_{L^2}+\big\|e^{\kappa |z|}(r^M_2-r^N_2)(t)\big\|_{L^2} \leq Ce^{-{c_1} \alpha^2_{\sharp}N^2} .$$
    \end{lem}

  \begin{proof}
 By \eqref{eqer}, 
    \begin{equation*} 
 i\partial_{t}(r^M_1-r^N_1) = \sum_{j=1}^3(q^M_j-q^N_j),
  \end{equation*}
  and we observe that the term $q_0$ does not depend on $N$ or $M$, thus $q_0^M=q_0^N$. We 
 compute 
     \begin{eqnarray*} 
\Big|  \frac{d}{dt} \int_\C e^{2 \kappa |z|}|r^M_1-r^N_1|^2dL \Big| &=&\Big| 2 \Re \int_\C e^{2 \kappa|z|}\ov{(r^M_1-r^N_1)} \partial_t (r^M_1-r^N_1)dL \Big|\\
      &\leq & C   \big\|e^{\kappa |z|}(r^M_1-r^N_1)\big\|_{L^2}   \sum_{j=1}^3   \|e^{\kappa |z|}(q^M_j-q^N_j)\|_{L^2}.
  \end{eqnarray*}
   Denote by 
  $$\theta(t)= \big\|e^{\kappa |z|}(r^M_1-r^N_1)(t)\big\|^2_{L^2}+\big\|e^{\kappa |z|}(r^M_2-r^N_2)(t)\big\|^2_{L^2},$$
   then we can prove that $\theta$ satisfies   the inequation
    \begin{equation} \label{ineq1}
  \big|  \frac{d}{dt} \theta(t) \big|   \leq C \theta(t),
      \end{equation}
where $C>0$ does not depend on $N,M>0$. To do  this, we can proceed as in the proof of Lemma~\ref{lem-borne}~: the estimates are the same, simply using 
   $$ \|  r^M_1     \|_{L^\infty}+ \|  r^N_1     \|_{L^\infty}+  \|  r^M_2     \|_{L^\infty}+ \|  r^N_2     \|_{L^\infty}  \leq C.$$
Next by Lemma~\ref{lem-borne}, we have, for any $c_1<\frac{c_0}2$
 $$\theta(N) =\big\|e^{\kappa |z|}r^M_1(N)\big\|^2_{L^2}+\big\|e^{\kappa |z|}r^M_2(N)\big\|^2_{L^2}  \leq Ce^{-2{c_1} \alpha^2_{\sharp}N^2}.$$
  By integration of~\eqref{ineq1} on $[t, N]$ we deduce that for all $0\leq t\leq N$
  $$  \theta(t) \leq \theta(N) +C\int_t^N  \theta(\sigma) d\sigma  \leq   Ce^{-2{c_1} \alpha^2_{\sharp}N^2}+C\int_t^N  \theta(\sigma) d\sigma. $$
  Therefore, by Lemma~\ref{lem-Gro}, for all $0\leq t\leq N$
    $$  \theta(t) \leq Ce^{-2{c_1} \alpha^2_{\sharp}N^2} +Ce^{-2{c_1} \alpha^2_{\sharp}N^2} \int_t^N e^{C \sigma } d\sigma \leq Ce^{-2{c_2} \alpha^2_{\sharp}N^2},$$
  for any $c_2<c_1$,  which was the claim.
    \end{proof}
    
    \subsection{Conclusion of the proof of Theorem \ref{thm-multi}}

For $\kappa>0$, we denote by 
$$\mathcal{X}^{\kappa}=\big\{u\in \mathscr{S}'(\C), \; e^{ \kappa |z|}u \in L^2(\C)\big\}.$$

 By Lemma~\ref{lem-cauchy},  for all $T>0$, the sequence $\big( r^M_1 , r^M_2 \big)_{M \geq 0}$ is a Cauchy sequence in the space $\mathcal{C} \big([0,T];   \mathcal{X}^{\kappa}  \times \mathcal{X}^{\kappa} )$, hence it converges in $\mathcal{C} \big([0,T];   \mathcal{X}^{\kappa}  \times \mathcal{X}^{\kappa} )$. By Lemma~\ref{lem-borne}, its limit satisfies the bound
    $$\|e^{\kappa |z|}r_1(t)\|_{L^2}+\|e^{\kappa |z|}r_2(t)\|_{L^2}  \leq Ce^{-{c_1} \alpha^2_{\sharp}t^2},$$
for any $c_1<\frac{c_0}2$  and  for all $t \geq 0$. \medskip
    
    Now let us prove that for all $t\in [0,T]$ we have $( r_1(t) , r_2(t) \big) \in   {\E}  \times  {\E}$, so that we will deduce that $( r_1 , r_2)  \in \mathcal{C} \big([0,T];   \mathcal{X}^{\kappa}_\E  \times \mathcal{X}^{\kappa}_\E )$. Fix $t\in [0,T]$. In the next lines, we do not mention the dependence on~$t$. For $j=1,2$,  write $r_j^M(z)=f_j^M(z) \e^{-|z|^2/2}$, where~$f_j^M$ is entire. By the Carlen inequality~\eqref{hyp}, for all $z\in \C$ and $M\geq 1$
$$|f_j^M(z) \e^{-|z|^2/2}  | \leq \|r_j^M\|_{L^{\infty}(\C)} \leq C \|r_j^M\|_{L^{2}(\C)} \leq C.$$
Therefore, for all $K>0$ and $M\geq 1$, we get 
$$|f_j^M(z)| \leq C_K, \quad |z|\leq K.$$
By the Montel theorem, there exists an entire function $f_j$ such that, when $M\longrightarrow +\infty$, up to a subsequence  $f^M_j \longrightarrow f_j$, uniformly on any compact of $\C$, and by uniqueness of the limit we have $r_j(z)=f_j(z)\e^{-|z|^2/2} \in \mathcal E $. \medskip

    To complete the proof of Theorem~\ref{thm-multi}, it remains to show that \eqref{bb} holds for all $k \in \N$. We proceed by induction on $k\geq 0$. The case $k=0$ has just been proven. Let $k\geq 0$ such that  
       \begin{equation}\label{bornej}
        \|e^{\kappa |z|}(\partial^j_t r_1)(t)\|_{L^2}+\|e^{\kappa |z|}(\partial^j_t r_2)(t)\|_{L^2}  \leq Ce^{-{c_1} \alpha^2_{\sharp}t^2},
          \end{equation}
        holds true for all $0 \leq j \leq k$, and where the constant $c_1<1/4$ can be chosen arbitrarily close to $1/4$. Then by \eqref{eqer}, 
           \begin{equation*} 
 i\partial^{k+1}_{t}r_1=  \sum_{j=0}^3\partial^k_t q_j.
  \end{equation*}
        Using the Leibniz rule, we observe that $\partial^k_t q_j$ is a trilinear term in $\dis (\partial^j_t r_\ell)_{0 \leq j \leq k}$,  $\dis (\partial^j_t X_\ell)_{0 \leq j \leq k}$, and $\dis (\partial^j_t Y_\ell)_{0 \leq j \leq k}$.  We write
                       \begin{equation*} 
       \big \|e^{\kappa |z|}(\partial^{k+1}_t r_1)(t)\big\|_{L^2}   \leq            \big\|e^{\kappa |z|}(\partial^{k}_t q_0)(t)\big\|_{L^2} +   \sum_{j=1}^3           \big\|e^{\kappa |z|}(\partial^{k}_t q_j)(t)\big\|_{L^2} .
    \end{equation*}
       To bound the first term, we can use  Lemma~\ref{lem-X} repeatedly with $c_0<1/2$ arbitrarily close to $1/2$, hence for all $c_1<1/2$ we get $  \big\|e^{\kappa |z|}(\partial^{k}_t q_0)(t)\big\|_{L^2}  \leq Ce^{-{c_1} \alpha^2_{\sharp}t^2}$. To bound the other terms, we use~\eqref{bornej} and Lemma~\ref{lem-X}, which implies  $\big\|e^{\kappa |z|}(\partial^{k+1}_t r_1)(t)\big\|_{L^2}     \leq Ce^{-{c_1} \alpha^2_{\sharp}t^2}$. With the same arguments we obtain   $\big\|e^{\kappa |z|}(\partial^{k+1}_t r_2)(t)\big\|_{L^2}     \leq Ce^{-{c_1} \alpha^2_{\sharp}t^2}$.\medskip
    
By the same manner,   when one of the traveling waves satisfies the bound~\eqref{gb1} for some $c_0<1/2$, one can establish \eqref{cc}. In this latter case, the constant $\wt{c_k}>0$ giving the rate of the Gaussian decay may depend on $k\in \N$. \medskip
    
    The relations~\eqref{form-cons} are obtained by plugging the expressions~\eqref{sol-thm} in the conservation laws and using Lemma~\ref{lem-exp}, together with the values given in \cite[equation~(1.15)]{Schw-Tho}.

    \subsection{Proof of the bound \eqref{conji}}\label{para-disp}

    By the change of unknown  $(\wt{u}, \wt{v}) =e^{-i \delta t H}(u,v)$, we have $(\wt{r_1}, \wt{r_2}) =e^{-i \delta t H}(r_1,r_2)$. For $s\geq 0$,
    $$  \|\<z\>^s (\partial^k_t \wt{r_1}) \|_{L^2}= \|\<z\>^s \partial^k_t ( e^{-i \delta tH}  {r_1}  )  \|_{L^2},  $$
    and by the Leibniz rule, we are reduced  to bound terms of the form $\|\<z\>^s H^j e^{-i \delta tH} (\partial^{\ell}_t     {r_1}  )  \|_{L^2} $ for $0 \leq j, \ell \leq k$. By~\eqref{equiNor}, 
              \begin{eqnarray*}
 \|\<z\>^s H^j (\partial^{\ell}_t     {r_1}  )  \|_{L^2} &\leq & C  \| H^j e^{-i \delta tH} (\partial^{\ell}_t     {r_1}  )  \|_{\HH^s}  \\
 &\leq & C  \|  \partial^{\ell}_t     {r_1}    \|_{\HH^{s+2j}} \\
  &\leq & C  \|\<z\>^{s+2j} \partial^{\ell}_t     {r_1}    \|_{L^2} \\
    &\leq & C  e^{-\wt{c_\ell} t^2},
  \end{eqnarray*} 
   where in the last line we used \eqref{cc}.
   
   %%%%%%%%%%%%%%%%%%%%%%%%%%%%%%%%%%%%%%%%%%%%%%%%%%%%%%%%%%%
    
        \section{Proof of Theorem \ref{thm-uni}} Consider $(\wt{u}, \wt{v}) =( X+\wt{r_1}, Y+\wt{r_2})\in \mathcal{C}   \big(\R , \mathcal{X}^{\kappa}_\mathcal{E} \times \mathcal{X}^{\kappa}_\mathcal{E}\big)$ a multi-soliton as given in Theorem~\ref{thm-uni}. We stress that for all $j \neq \ell$, we have $\alpha_j \neq \alpha_{\ell}$ and that for all $1 \leq j \leq n$, $\alpha_j \neq 0$.\medskip
        
        Similarly to \eqref{travj}, we assume that
        \begin{equation*}  
\begin{aligned} 
&X(t,z)=  \sum_{j=1}^n e^{-i\lambda_j t}U_j(z+\alpha_j t) e^{\frac{1}{2}(\overline z \alpha_j - z \overline{\alpha_j})t}\\[4pt]
&Y(t,z) = \sum_{j=1}^n e^{-i\mu_j t}V_j(z+\alpha_j t) e^{\frac{1}{2}(\overline z \alpha_j - z \overline{\alpha_j})t}  \;,
\end{aligned}
\end{equation*}
and we set $K_j:=\|U_j\|_{L^2}= \|V_j\|_{L^2}$.\medskip

        \underline{Step 1 :  Exponential decay of the error.} In this paragraph only, we write  $r_1=\wt{r_1}$ and $r_2=\wt{r_2}$. Assume that, when $t \longrightarrow 0$,
     \begin{equation}\label{lim}
   \|e^{\kappa |z|}r_1(t)\|_{L^2}+\|e^{\kappa |z|}r_2(t)\|_{L^2}   \longrightarrow 0.
       \end{equation}
    Let us show that  for all $t \geq 0$
                   \begin{equation} \label{cont-R}
  \|r_1\|_{L^2(\C)}^2+  \|r_2\|_{L^2(\C)}^2 \leq C_{\kappa}  e^{- {\kappa \alpha_{min}t}}.
  \end{equation}
  Starting from the relation
               \begin{equation*} 
 i\partial_{t}r_1=  \sum_{j=0}^3 q_j ,
  \end{equation*}
     similarly to \eqref{sim} we compute 
     \begin{eqnarray*} 
\Big|  \frac{d}{dt} \int_\C |r_1|^2dL \Big|   &=& \Big|2 \Im \int_\C  \ov{r_1}  \big(\sum_{j=0}^3q_j\big) dL \Big|\\
      &\leq & 2   \|  r_1\|_{L^2}   \sum_{j=0}^2   \|  q_j\|_{L^2}  ,
  \end{eqnarray*}
and we observe that  the contribution of $q_3$ cancels in the previous line. We now control each term~$ \|  q_j\|_{L^2} $, for $0\leq j \leq 2$. \medskip

  $\bullet$ Control of $\|q_0 \|_{L^2} $. We have already   controlled this term, namely by~\eqref{control-q0} (with $\kappa=0$)
   \begin{equation*} 
   \| q_0\|_{L^2}  \leq C   e^{-{c} \alpha^2_{\sharp}t^2}.
  \end{equation*}~

  $\bullet$ Control of $\|q_1 \|_{L^2} $. We directly obtain
                  \begin{eqnarray}
\|q_1\|_{L^2}  &\leq & C \|r_1 Y^2\|_{L^2}+ C \|r_2 XY\|_{L^2} \nonumber \\
 &\leq & C \big(  \|e^{\kappa |z|}r_1 \|_{L^2}+ \|e^{\kappa |z|}r_2 \|_{L^2}\big)  \|e^{-\kappa |z|}Y \|_{L^\infty} \big(\| Y \|_{L^\infty}  +\| X \|_{L^\infty} \big) \nonumber \\
  &\leq & C \big(  \|e^{\kappa |z|}r_1 \|_{L^2}+ \|e^{\kappa |z|}r_2 \|_{L^2}\big)  \|e^{-\kappa |z|}Y \|_{L^\infty}. \label{q1}
    \end{eqnarray} 
We now use the expression of $Y$ and \eqref{B4}, and we rely on the crucial fact that  $\alpha_j\neq 0$ for all $1 \leq j \leq n$ : 
denote by $\alpha_{min}=\dis \min_{1 \leq j \leq n} |\alpha_j|$, then there exists a universal constant $c_1>0$ such  that 
         \begin{eqnarray*}
  \|e^{-\kappa |z|}Y \|_{L^\infty}& \leq & C   \big(e^{- {\kappa \alpha_{min}t}/2}+ e^{- c_1 \alpha^2_{min}t^2/4}  \big) \\
& \leq & C  e^{- {\kappa \alpha_{min}t}/2}.
      \end{eqnarray*} 
      Therefore, from \eqref{q1} and \eqref{lim} we deduce 
                    \begin{eqnarray*}
\|q_1\|_{L^2}  &\leq&  C  e^{- {\kappa \alpha_{min}t}/2} \big(  \|e^{\kappa |z|}r_1 \|_{L^2}+ \|e^{\kappa |z|}r_2 \|_{L^2}\big) \\
&\leq & C  e^{- {\kappa \alpha_{min}t}/2} .
  \end{eqnarray*}

       $\bullet$ Control of $\|q_2\|_{L^2} $. Similarly we get 
                         \begin{eqnarray*}
\|q_2 \|_{L^2}  &\leq & C \big(  \|e^{\kappa |z|}r_1 \|_{L^2}+ \|e^{\kappa |z|}r_2 \|_{L^2}\big)  \big(   \|e^{-\kappa |z|}X \|_{L^\infty}+  \|e^{-\kappa |z|}Y \|_{L^\infty}\big) \big(\| r_1 \|_{L^\infty}  +\| r_2 \|_{L^\infty}  \big)\\
 &\leq  &C \big(  \|e^{\kappa |z|}r_1 \|_{L^2}+ \|e^{\kappa |z|}r_2 \|_{L^2}\big)  \big(   \|e^{-\kappa |z|}X \|_{L^\infty}+  \|e^{-\kappa |z|}Y \|_{L^\infty}\big)  \\
 & \leq &C  e^{- {\kappa \alpha_{min}t}/2} .
  \end{eqnarray*}

           Putting the previous estimates together we get

             \begin{equation*} 
\Big|  \frac{d}{dt}  \big( \|r_1\|_{L^2(\C)}^2+  \|r_2\|_{L^2(\C)}^2\big) \Big| \leq C   \big( \|r_1\|_{L^2(\C)}^2+  \|r_2\|_{L^2(\C)}^2\big)^{1/2} e^{- {\kappa \alpha_{min}t}/2}
  \end{equation*}
  and by integration on $[t, +\infty)$, using \eqref{lim}, we deduce  \eqref{cont-R}. \medskip

             \underline{Step 2 :  An explicit $L^{\infty}$ bound.}  Consider two multi-solutions
            $$(u, v) =( X+r_1, Y+r_2), \qquad (\wt{u}, \wt{v}) =( X+\wt{r_1}, Y+\wt{r_2}),$$
            where the remainder terms satisfy \eqref{lim}. 
            Then by   \eqref{eqer}, the errors satisfy the equations
           \begin{equation*} 
 i\partial_{t}r_1=  \sum_{j=0}^3  q_j,\qquad  i\partial_{t}\wt{r_1}=  \sum_{j=0}^3 \wt{q_j}.
  \end{equation*}
    Denote by $\rho_j=r_j- \wt{r_j}$ and set
    $$\theta(t) =  \|\rho_1(t)\|^2_{L^2(\C)}+\|\rho_2(t)\|^2_{L^2(\C)} .$$
    Observe that, thanks to \eqref{cont-R}, we already have the bound
                       \begin{equation} \label{cont-th}
  \theta(t)\leq C_{\kappa}   e^{- {\kappa \alpha_{min}t}}.
  \end{equation}
  Denote by 
          \begin{equation*}  
  G(t)=  \| X(t) \|^2_{L^{\infty}}+  \| Y(t) \|^2_{L^\infty}+ \| r_1 (t)\|^2_{L^{\infty}} +\| r_2(t) \|^2_{L^{\infty}} +\| \wt{r_1} (t)\|^2_{L^{\infty}}  +\| \wt{r_2} (t)\|^2_{L^{\infty}},
  \end{equation*}
  and set $K_{max}=\dis \max_{1 \leq j \leq n} K_j$. We now show that there exists a universal constant $c_0>0$ and $t_0>0$ such that for all $t\geq t_0$, 
           \begin{equation} \label{gk}
  G(t) \leq c_0 K^2_{max}. 
    \end{equation}
  Since $\alpha_j \neq \alpha_{\ell}$, one has $\dis \| X(t) \|_{L^{\infty}}  \longrightarrow \max_{1\leq j \leq n}\|U_j\|_{L^{\infty}}$, when $t \longrightarrow +\infty$. Besides, by \eqref{hyp},  $\|U_j\|_{L^{\infty}} \leq C \|U_j\|_{L^{2}}=CK_j$. Therefore, for $t$ large enough, $\dis \| X(t) \|_{L^{\infty}}  \leq 2C  K_{max}$. We proceed similarly for $Y$ and we can use~\eqref{ass-exp} to conclude that~\eqref{gk} holds true. \medskip

  \underline{Step 3 :  A differential inequality.} 
   Let us show that there exists an universal constant $C_0>0$  such that for all  $t\geq t_0$
             \begin{equation} \label{derth}
\big|  \frac{d}{dt} \theta(t) \big|  \leq C_0   K^2_{max} \theta(t).
  \end{equation}

  As in \eqref{sim} we compute 
         \begin{equation} \label{der}
\Big|  \frac{d}{dt} \|\rho_1\|^2_{L^2} \Big|  \leq C   \|\rho_1\|_{L^2}   \sum_{j=1}^3   \|q_j-\wt{q_j}\|_{L^2},
  \end{equation}
  where we observe that $q_0=\wt{q_0}$.  Using the expressions \eqref{exp-q} we bound the previous terms. In the sequel, we assume that $t\geq t_0$.\medskip
    
    $\bullet$ Control of $\|q_1-\wt{q_1}\|_{L^2} $. Using~\eqref{gk} we directly obtain
                  \begin{eqnarray*}
\|q_1-\wt{q_1}\|_{L^2}  &\leq & C \|\rho_1 Y^2\|_{L^2}+ C \|\rho_2 XY\|_{L^2}\\
 &\leq & C \big(  \|\rho_1 \|_{L^2}+ \| \rho_2 \|_{L^2}\big)   \big(\| X \|^2_{L^\infty}  +\| Y \|^2_{L^\infty} \big)\\
 &\leq & C \big(  \|\rho_1 \|_{L^2}+ \| \rho_2 \|_{L^2}\big) K^2_{max}.
    \end{eqnarray*} ~

       $\bullet$ Control of $\|q_2-\wt{q_2}\|_{L^2} $. Similarly we get 
                         \begin{eqnarray*}
\|q_2-\wt{q_2}\|_{L^2}  &\leq &
 C \big(  \|\rho_1 \|_{L^2}+ \| \rho_2 \|_{L^2}\big)  \big(   \| X \|_{L^\infty}+  \| Y \|_{L^\infty}\big) \big(\| r_1 \|_{L^\infty}  +\| r_2 \|_{L^\infty} +\| \wt{r_1} \|_{L^\infty}  +\| \wt{r_2} \|_{L^\infty}\big)\\
 &\leq  &C \big(  \| \rho_1 \|_{L^2}+ \| \rho_2 \|_{L^2}\big) K^2_{max}.
  \end{eqnarray*}~

           $\bullet$ 
           Control of $\|q_3-\wt{q_3}\|_{L^2} $. By the same manner, we have
               \begin{eqnarray*}
\|q_3-\wt{q_3}\|_{L^2}  &\leq &   C \big( \|\rho_1 \|_{L^2}+ \| \rho_2 \|_{L^2}\big)    \big(\| r_1 \|_{L^\infty}  +\| r_2 \|_{L^\infty} +\| \wt{r_1} \|_{L^\infty}  +\| \wt{r_2} \|_{L^\infty}\big)^2\\
  &\leq &    C \big( \|\rho_1 \|_{L^2}+ \| \rho_2 \|_{L^2}\big) K^2_{max}    .
    \end{eqnarray*} ~
    
       Therefore, by \eqref{der} and the previous estimates
             \begin{equation*}  
\Big|  \frac{d}{dt} \|\rho_1\|^2_{L^2} \Big|  \leq C   ( \|\rho_1 \|^2_{L^2}+ \| \rho_2 \|^2_{L^2}\big)K^2_{max}.
  \end{equation*}
  The same bound holds for $\dis \Big|  \frac{d}{dt} \|\rho_2\|^2_{L^2} \Big|  $, and we deduce \eqref{derth}. \medskip
  
  \underline{Step 4 : Backward  Gr\"onwall.}  Let $t_0 \leq t \leq M$. We integrate \eqref{derth} on $[t,M]$
                 \begin{equation*}
\theta(t)  \leq    \theta(M)+C_0 K^2_{max} \int_{t}^M \theta(s)ds      .
    \end{equation*}
    We are able to apply Lemma~\ref{lem-Gro}    and get for all $0 \leq t \leq M$
\begin{eqnarray*}
\theta(t) & \leq &   \theta(M)+C_0K^2_{max}\theta(M) \int_{t}^M \exp\big(C_0 K^2_{max}\int_t^{\sigma}d\tau \big) d\sigma    \\
& \leq &   \theta(M)+C_0K^2_{max}\theta(M) \int_{t}^M e^{C_0K^2_{max} \sigma} d\sigma \\
& \leq &   \theta(M) \big(1+  e^{C_0 K^2_{max} M }  \big).
    \end{eqnarray*}
    Next, by \eqref{cont-th}
                    \begin{equation*}  
\theta(t) \leq C  e^{- {\kappa \alpha_{min}M}} \big(1+  e^{C_0 K^2_{max} M}  \big),
 \end{equation*}
    which tends to 0 when $M\longrightarrow +\infty$, provided that $\kappa >c_0 K^2_{max}/\alpha_{min}$ is chosen large enough. As a conclusion $\theta(t)=0$ for all $t\geq t_0$ which in turn implies that $\theta \equiv 0$ on $\R$, since equation~\eqref{eq1} is globally well-posed on~$\R$. In the case where the traveling waves take the form \eqref{prog}, one has 
    $|\alpha| = \frac{\sqrt{3}}{32\pi}K^2$, and \eqref{dela} follows.

        %%%%%%%%%%%%%%%%%%%%%%%%%%%%%%%%%%%%%%%%%%%%%%%%%%%%%%%%%%%

          \section{Proof of Theorem \ref{thm-super}} 
   
        The proof of this result is in the same spirit as the proof of Theorem~\ref{thm-multi}, but here the error estimate will be done   starting from $t=0$ instead of considering times $t\gg 1$. \medskip

Assume that  $(U_0,V_0)\in  \E \times \E$ is a traveling wave solution to \eqref{eq1} in the sense~\eqref{trav}, and that there exist $c_0,C>0$ such that 
\begin{equation*} 
|U_0(z)|+|V_0(z)| \leq C e^{-c_0 |z|^2}.
\end{equation*}
We moreover assume that $\|U_0\|_{L^2}= \|V_0\|_{L^2}=1$, and we denote by $\alpha_0 \in \C$ the speed of this traveling wave and consider $(\lambda_0, \mu_0)$  the phase parameters. Next, for $a_j, b_j,\gamma_j \in \C$, $K\geq 0$, and $\theta \in \R$, we define 
\begin{equation*} 
 U_j=  K e^{i a_j}R_{\gamma_j}L_{\theta}U_0, \qquad  V_j= K e^{i b_j} R_{\gamma_j} L_{\theta} V_0   \;.
\end{equation*}
By~\cite[Proposition~1.8 $(iv)$]{Schw-Tho}, each couple $(U_j,V_j)$ defines a traveling wave
\begin{equation*} 
\begin{aligned} 
& X_j(t,z)= e^{-i\lambda_j t}U_j(z+\alpha t) e^{\frac{1}{2}(\overline z \alpha - z \overline{\alpha})t}\\[4pt]
&   Y_j(t,z)= e^{-i\mu_j t}V_j(z+\alpha t) e^{\frac{1}{2}(\overline z \alpha- z \overline{\alpha})t}  \;,
\end{aligned}
\end{equation*}
with speed $\alpha= \alpha_0 K^2 e^{-i \theta }$ and where $(\lambda_j, \mu_j)= K^2\big(\lambda_0+2 \Im( \ov{\alpha_0} \gamma_j e^{i \theta}), \mu_0+2 \Im( \ov{\alpha_0} \gamma_j e^{i \theta})\big)$. 
We consider the solution $(u,v)$ to \eqref{eq2}
 \begin{equation*}
 u=\sum_{j=1}^n X_j+r_1:=X+r_1,\qquad v=\sum_{j=1}^n Y_j+r_2:=Y+r_2,
 \end{equation*}
 such that  $(r_1(0), r_2(0))=(0,0)$. We now have to estimate the error term $(r_1,r_2)$ and by reversibility of the equation it is enough to consider the case $t\geq 0$.\medskip

We write the expansion \eqref{eqer}--\eqref{exp-q}, and similarly to \eqref{sim} we obtain
     \begin{equation*} 
   \frac{d}{dt} \int_\C  |r_1|^2dL  \leq  C   \| r_1\|_{L^2}   \sum_{j=0}^2   \|  q_j\|_{L^2}. 
  \end{equation*}
We now    estimate each term $\| q_j\|_{L^2} $. Denote by $\eta(t):= \|r_1(t)\|_{L^2(\C)}^2+  \|r_2(t)\|_{L^2(\C)}^2$.
\medskip

$\bullet$ Control of $\| q_0\|_{L^2} $.  For all $1 \leq j \leq n$, $ i\partial_{t}X_j=  \Pi \big(|Y_j|^2 X_j\big) $, thus
$$q_0=\sum_{\substack{1 \leq j, k, \ell  \leq n \\ (j,k, \ell) \neq (j,j,j) }} \Pi \big(Y_j \ov{Y_k}X_{\ell} \big).   $$
Assume for instance that $j\neq k$, so that $\gamma_j \neq \gamma_k$. Then by \eqref{conti2} 
   \begin{eqnarray*} 
\|  \Pi \big(Y_j \ov{Y_k}X_{\ell} \big) \|_{L^2}&\leq &C\|   Y_j {Y_k}X_{\ell}  \|_{L^2} \\
&= &CK^3 \|  (R_{\alpha t+\gamma_j}V_0 ) (R_{\alpha t+\gamma_k}V_0)  (R_{\alpha t+\gamma_\ell}U_0)  \|_{L^2}\\
&= &CK^3 \|   (R_{\gamma_j-\gamma_\ell }V_0 ) (R_{\gamma_k-\gamma_\ell }V_0) U_0  \|_{L^2} \\
&\leq &C  K^3  \| (R_{\gamma_j-\gamma_\ell }V_0 ) (R_{\gamma_k-\gamma_\ell }V_0) \|_{L^{\infty}}  \| U_0  \|_{L^2} .
  \end{eqnarray*}
  Next by~\eqref{B1} 
     \begin{equation*} 
\|  \Pi \big(Y_j \ov{Y_k}X_{\ell} \big) \|_{L^2}
\leq  C K^3 e^{ -\frac{c_0}2|\gamma_j-\gamma_k|^2}\|   U_0  \|_{L^2} \\
= C  K^3  e^{ -\frac{c_0}2|\gamma_j-\gamma_k|^2}.
  \end{equation*}
The other terms are treated similarly, as a consequence   
   \begin{equation*} 
   \|q_0\|_{L^2}  \leq C  K^3  e^{- \frac{c_0}2\eps^{-2}}.
  \end{equation*}~
  
$\bullet$ The controls of  $\| q_j\|_{L^2} $ for $1\leq j\leq 2$ are obtained as in the proof of Theorem~\ref{thm-multi}, and we get
   \begin{equation*} 
\|   q_j \|_{L^2}\leq C   n^2K^2\eta^{1/2}.
  \end{equation*}~

  Putting the previous estimates together we can write 
             \begin{equation*}
  \frac{d}{dt} \eta(t)  \leq C  \eta(t)^{1/2}\big(n^2K^2\eta(t)^{1/2}+ K^3  e^{- \frac{c_0}2 \eps^{-2}}\big) \leq C n^2K^2\eta(t)+ CK^4  e^{-{c_0} \eps^{-2}},
  \end{equation*}
 and by integration
              \begin{equation*}
 \eta(t)  \leq C n^2K^2\int_0^t\eta(s)ds+ CK^4 t e^{-{c_0} \eps^{-2}}.
  \end{equation*}
  Finally, the Gr\"onwall estimate implies 
               \begin{equation*}
 \eta(t) = \|r_1(t)\|_{L^2(\C)}^2+  \|r_2(t)\|^2_{L^2(\C)}\leq  CK^4 t e^{-{c_0} \eps^{-2}+  C n^2K^2t}.
  \end{equation*}
In particular, when $c_0=1/2$ we obtain \eqref{boundr}.

    %%%%%%%%%%%%%%%%%%%%%%%%%%%%%%%%%%%%%%%%%%%%%%%%%%%%%%%%%%%

\section{Proof of Theorem \ref{thm-lin} }

We will adopt the formalism of  \cite[Section~7]{Faou-Rapha} so that the result of Theorem~\ref{thm-lin} will be a direct application of~\cite[Proposition~7.1]{Faou-Rapha}.

As in \cite{Faou-Rapha} we denote by $\mathcal{T}$ the CR trilinear operator which was first defined in \cite{FGH} and further studied in \cite{GHT1,GHT2}. This operator~$\mathcal{T}$ is defined by 
$$(u_1,u_2,u_3) \mapsto \mathcal{T}(u_1,u_2,u_3)(w):=\int_{\R^2}\int_{\R}u_1(x+w)\ov{u_2(x+\lambda x^\perp +w)}u_3(\lambda x^\perp+w)d\lambda dx,$$ 
where for $x=(x_1,x_2)\in \R^2$ we have set $x^\perp=(-x_2,x_1)$. By \cite[Lemma 8.2]{GHT1}, when it is restricted to the Bargmann-Fock space $\E$, the operator $\mathcal{T}$ can be simply expressed using $\Pi$ : for all $u_1,u_2,u_3 \in \E$, 
\begin{equation*}
\mathcal{T}(u_1,u_2,u_3 )= {\pi^2} \Pi\big(u_1 \ov{u_2} u_3 \big).
\end{equation*}
 
Next, following \cite{Faou-Rapha}, we define 
$$\mathcal{T}[F]u:=\mathcal{T}(F,F,u),$$
so that $\mathcal{T}[F]u=  {\pi^2} \Pi\big(|F|^2 u \big)$ when $F,u\in \E$.
Consider the solution $(u,v)$ given by Theorem \ref{thm-multi},  define $\wt{u}(s,z) := {u}(e^{ s},z)$ and $\dis F(s,z) := \frac{e^{ s/2}}{\pi} {v}(e^{ s},z)$. Then 
$$i \partial_s \wt{u}=\pi^2 \Pi(|F|^2 \wt{u})=\mathcal{T}[F] \wt{u},$$
for all $s\in \R$. Recall the definition~\eqref{def-sobo} of the Sobolev space $\HH^s(\C)$. Using the explicit representation~\eqref{sol-thm}, we observe that for all $\sigma\geq 0$ and $k \geq 0$ we have the bounds
$$\| \partial^k_s F(s)\|_{\HH^{\sigma}(\C)}+ \| \partial^k_s \wt{u}(s)\|_{\HH^{\sigma}(\C)} \leq C e^{c s},\quad \forall s \geq 0.$$
Moreover, when $s\longrightarrow +\infty$
$$\|   \wt{u}(s)\|_{\HH^1(\C)}  \sim  C e^s.$$
Therefore,~\cite[Proposition~7.1]{Faou-Rapha} can be applied: we set 
$$V(t ):=\frac{1}{t \ln t}  \big|e^{-it H} F(\ln \ln t)\big|^2,$$
which satisfies \eqref{potendecay}. Next by~\cite[Proposition~7.1]{Faou-Rapha}, there exists $r_0  \in \mathcal{C}\big(\R  ; \HH^1(\C)\big)$ which satisfies 
  \begin{equation*} 
  \|{r_0}(t)   \|_{\HH^1(\C)} \longrightarrow 0, \qquad t \longrightarrow +\infty,
  \end{equation*}
and such that 
$$\psi(t):=e^{-it H}  \wt{u}( \ln \ln t)+ r_0(t)= e^{-it H}  {u}(  \ln t)+ r_0(t)$$
is solution to the equation \eqref{harm}. Let us give a better description of $\psi$. 
By \eqref{sol-thm} we have
$$\wt{u}(\ln \ln t)= {u}(  \ln t)= \sum_{j=1}^n e^{-i\lambda_j \ln t}  R_{\alpha_j \ln t }U_j+r_1(\ln t),$$
thus 
    \begin{eqnarray*} 
 \psi(t)&=&e^{-it H}  \sum_{j=1}^n e^{-i\lambda_j \ln t}  R_{\alpha_j \ln t }U_j+ \big(e^{-it H} r_1(\ln t)+r_0(t)\big)\\
 &=&  \sum_{j=1}^n e^{-i\lambda_j \ln t}e^{-2i t} L_{-2  t} R_{\alpha_j \ln t }U_j  +\eta(t),
    \end{eqnarray*} 
where $\eta(t):=e^{-it H} r_1(\ln t)+r_0(t) $ satisfies 
  \begin{equation*} 
  \|\eta(t)   \|_{\HH^1(\C)} \longrightarrow 0, \qquad t \longrightarrow +\infty.
  \end{equation*}
This completes to proof of Theorem \ref{thm-lin}.

    %%%%%%%%%%%%%%%%%%%%%%%%%%%%%%%%%%%%%%%%%%%%%%%%%%%%%%%%%%%

  \appendix
  
    \section{On the decay of stationary solutions}
  
  In this section, we show that any stationary solution $(u(t),v(t))=(e^{-i \lambda t}U, e^{-i \mu t}V)$ to \eqref{eq0} with $(U,V) \in \E \times \E$ has a Gaussian decay. Let $\lambda,\mu \in \R$, $\sigma \in\{-1,1\}$ and consider the system
   \begin{equation} \label{eqs}
\left\{
\begin{aligned}
& \lambda U= \Pi(|V|^2 U)     \\
& \mu V= \sigma \Pi(|U|^2 V). 
\end{aligned}
\right.
\end{equation}
Then we have a natural extension of~\cite[Theorem~5.3]{GGT} :
  
  \begin{thm}\label{thmDec}
Let $(U, V) \in \mathcal E$ be a solution of \eqref{eqs}. Then, for any
$$
\eta>\eta_0 = \left( \frac{1}{2} + \frac{1}{2} \frac{\log 2}{\log 3} \right)^{-1} \sim 1.226\dots,
$$
the following estimates hold true,
\begin{equation}\label{decay1}
\ |U(z)|   \leq C_\eta e^{|z|^\eta-\frac{1}{2}|z|^2} , \qquad |V(z)| \leq C_\eta e^{|z|^\eta-\frac{1}{2}|z|^2}, \qquad  \forall z\in \C.
 \end{equation}
\end{thm}

It is classical that a bound of the form \eqref{decay1} gives an estimate of the number of zeros of the corresponding function. More precisely, as proven in \cite[Corollary~5.5]{GGT}, if one denotes by 
 $$
N(R) = \# \big\{ z \in \mathbb{C} \; \mbox{such that} \;\, U(z) = 0 \; \mbox{and} \; |z|<R \big\} ,
$$
then for any $\eta > \eta_0$, 
$$
\frac{N(R)}{R^\eta} \longrightarrow 0 \quad \mbox{ as } \;\;R \longrightarrow +\infty,
$$
and similarly for $V$. 

\begin{proof}
The argument follows the main lines of~\cite[Theorem~5.3]{GGT} where a similar result is established for the solutions $U_0\in \E$ of the equation 
$$
 \lambda U_0= \Pi(|U_0|^2 U_0).
$$
There are very few changes in the proof, and we just give the main steps of the argument. \medskip

We write the expansion $\dis U=\sum_{n=0}^{+\infty} c_n \phi_n$ and $\dis V=\sum_{n=0}^{+\infty} d_n \phi_n$.\medskip

\underline{Step 1 (Step 1 in \cite[paragraph 5.3.]{GGT}):} For $0<\kappa <1$, we set $M_n = \sup_{|w|> \kappa^{-n}} \big( |U(w)|+|V(w)|\big) $ and we prove 
$$
M_n \leq C_0 e^{-\frac{(1-\kappa)^2}{3} \kappa^{-2n}} + C_0 M_{n-1}^3,
$$
for some constant $C_0>0$. By an induction argument we show that there exists $\sigma>0$ such that 
\begin{equation} \label{eqa}
|U(z)|+ |V(z)| \leq C e^{-\sigma |z|^2}.
\end{equation}~

\underline{Step 2 (Step 2 in \cite[paragraph 5.3.]{GGT}):} The estimate~\eqref{eqa} implies that there exists $0<r<1$ such that
\begin{equation} \label{eqb}
|c_k| + |d_k| \leq  C r^k.
\end{equation}~

\underline{Step 3 (Steps 1 and 2  in \cite[paragraph 5.2.]{GGT}):}  In the coordinates $(c_n), (d_n)$, the system~\eqref{eqs} reads 
   \begin{equation*}  
\left\{
\begin{aligned}
& \lambda  c_k = \frac{1}{2 \pi} \sum_{\substack{ \ell,m,n \geq 0 \\ k + \ell = m + n }} \frac{(k+\ell)!}{2^{k+\ell} \sqrt{k! \ell ! m! n!}} \overline{d_\ell} d_m c_n, \quad k\geq 0   \\
&   \mu  d_k = \frac{\sigma }{2 \pi} \sum_{\substack{ \ell,m,n \geq 0 \\ k + \ell = m + n }} \frac{(k+\ell)!}{2^{k+\ell} \sqrt{k! \ell ! m! n!}} \overline{c_\ell} c_m d_n, \quad k\geq 0.
\end{aligned}
\right. 
\end{equation*}
With a bootstrap argument, starting from \eqref{eqb}, we show that for any $\gamma <\gamma _0=\frac{\log2}{2\log3}$ we have 
\begin{equation} \label{eqc}
|c_k| + |d_k| \leq  C k^{-\gamma k}.
\end{equation}~

\underline{Step 4 (Step 3 in \cite[paragraph 5.3.]{GGT}):} The estimate \eqref{eqc} implies that for all $\gamma <\gamma _0$
$$
|U(z)|+ |V(z)|  \leq C  e^{C | z|^{\delta}-\frac{1}{2} |z|^2},
$$
with $\delta = (\frac{1}{2} + \gamma)^{-1}$.
\end{proof}

  \section{Technical results}
  
  We reproduce a backward Gr\"onwall estimate taken from \cite[Lemma B.1]{Faou-Rapha}.
  
    \begin{lem}\label{lem-Gro}
  Let $t_0>0$ and $M>0$. Assume that $\beta>0$ and $\alpha>0$  are functions defined  on $(t_0,M)$, and that $F$ satisfies for all $t\in (t_0,M)$
  $$F(t) \leq \alpha(t)+\int_t^M \beta(\sigma) F(\sigma) d \sigma.$$
  Then for all $t\in (t_0,M)$
    $$F(t) \leq \alpha(t)+\int_t^M \alpha(\sigma) \beta(\sigma) \exp \big(\int_t^\sigma \beta(\tau) d\tau\big)    d \sigma.$$
  \end{lem}


\begin{thebibliography}{99}

 

\bibitem{ABD} 
A. Aftalion, X. Blanc, and J. Dalibard. 
 \newblock  Vortex patterns in a fast rotating Bose-Einstein condensate.
  \newblock{\em  Physical Review} A 71 (2005), 023611.


 \bibitem{ABN} 
 A. Aftalion, X. Blanc, and F. Nier. 
  \newblock  Lowest Landau level functional and Bargmann spaces for Bose-Einstein condensates.
   \newblock{\em J. Functional Anal.} 241 (2006), 661--702.
 
 
%\bibitem{AftaSerfa} 
%A. Aftalion and S. Serfaty. 
% \newblock  Lowest Landau level approach in superconductivity for the Abrikosov lattice close to~$H_{c_2}$.
%  \newblock{\em  Selecta Math.} (N.S.) 13 (2007), no. 2, 183--202.

%
% \bibitem{BiBiCrEv}
%A. Biasi, P. Bizon,  B. Craps, and O. Evnin.   
%\newblock   Two infinite families of resonant solutions for the Gross-Pitaevskii equation. 
%\newblock{\em  Phys. Rev. E} 98 (2018), no. 3, 032222, 12 pp. 
%
%
  \bibitem{BiBiCrEv2}
 A. Biasi, P. Bizon, B. Craps, and O. Evnin.
 \newblock   Exact lowest-Landau-level solutions for vortex precession in Bose-Einstein condensates. 
 \newblock{\em   Phys. Rev. A} 96, 053615 (2017).
% 
% 
%  \bibitem{BiBiEv}
%A. Biasi, P. Bizon,  and O. Evnin.
%\newblock  Solvable cubic resonant systems. 
%\newblock{\em Comm. Math. Phys.} 369 (2019), no. 2, 433--456.
%
%
%  \bibitem{Biasi-Evnin}
%A. Biasi  and O. Evnin.
%\newblock   Turbulent cascades in a truncation of the cubic Szeg\"o equation and related systems.
%\newblock{\em arXiv:2002.07785}
%

\bibitem{Carlen} 
E. Carlen.
\newblock  Some integral identities and inequalities for entire functions and their application to the coherent state transform. 
\newblock {\em J. Funct. Anal.} 97 (1991), no. 1, 231--249. 


\bibitem{Car-Galla} 
R. Carles and I.  Gallagher.
\newblock   Universal dynamics for the defocusing logarithmic Schr\"odinger equation.
\newblock {\em  Duke Math. J.} 167 (2018), no. 9, 1761--1801. 
 

%\bibitem{CKSTT}
%J. Colliander, M. Keel, G. Staffilani, H. Takaoka, and T. Tao. 
%\newblock Transfer of energy to high frequencies in the cubic defocusing nonlinear Schrdinger equation.
%\newblock {\em    Invent. Math.} 181 (2010), no. 1, 39--113. 

 \bibitem{CoteLeCoz}
R. C\^ote and S. Le Coz.
\newblock High-speed excited multi-solitons in nonlinear Schrdinger equations.
\newblock {\em    J. Math. Pures Appl.} (9) 96 (2011), no. 2, 135--166.

 \bibitem{CoteMarMer}
R. C\^ote, Y. Martel, and F. Merle.
\newblock Construction of multi-soliton solutions for the $L^2$-supercritical gKdV and NLS equations.
\newblock {\em     Rev. Mat. Iberoam.} 27 (2011), no. 1, 273--302. 
 
 
\bibitem{Clerck-Evnin}
M. De Clerck and O. Evnin.
\newblock Time-periodic quantum states of weakly interacting bosons in a harmonic trap.
\newblock {\em     Phys. Lett. A} 384 (2020), no. 36, 126930, 11 pp.


  \bibitem{DeleCozWeis}
F.  Delebecque, S. Le Coz, and  R. Weish\"aupl.
\newblock Multi-speed solitary waves of nonlinear Schrdinger systems: theoretical and numerical analysis. 
\newblock {\em    Commun. Math. Sci.} 14 (2016), no. 6, 1599--1624. 




\bibitem{FGH} 
E. Faou, P. Germain and Z. Hani. 
\newblock The weakly nonlinear large box limit of the 2D cubic NLS
 \newblock{\em  J. Amer. Math. Soc.} 29 (2016), no. 4, 915--982. 


\bibitem{Faou-Rapha} 
E. Faou and P. Rapha\"el.
\newblock On weakly turbulent solutions to the perturbed linear harmonic oscillator.
 \newblock{\em  Preprint : arXiv: 2006.08206}.
 
 \bibitem{Fer2}
G. Ferriere.
\newblock  Existence of multi-solitons for the focusing logarithmic non-linear Schr\"odinger equation.
\newblock{\em  Ann. Inst. H. Poincar Anal. Non Linaire} 38 (2021), no. 3, 841--875. 


 \bibitem{Fer1}
G. Ferriere. 
\newblock  The focusing logarithmic Schr\"odinger equation: analysis of breathers and nonlinear superposition.
\newblock{\em  Discrete Contin. Dyn. Syst.} 40 (2020), no. 11, 6247--6274. 

 
\bibitem{GGT}
P.  G\'erard, P. Germain, and L. Thomann. 
\newblock On the cubic lowest Landau level equation.
\newblock{\em Arch. Ration. Mech. Anal.} 231 (2019), no. 2, 1073--1128. 


%\bibitem{GG3} 
%P. G\'erard and Grellier. 
% \newblock The cubic Szeg\H{o} equation and Hankel operators.
% \newblock{\em Ast\'erisque} 389 (2017).


\bibitem{GHT1} 
P. Germain, Z. Hani, and L. Thomann. 
\newblock  On the continuous resonant equation for NLS. I. Deterministic analysis.
\newblock{\em   J. Math. Pures Appl. }  105 (2016), no. 1,  131--163.

\bibitem{GHT2} 
P. Germain, Z. Hani and L. Thomann. 
\newblock   On the continuous resonant equation for NLS. II. Statistical study.
\newblock{\em  Anal.~\&~PDE.} 8-7 (2015), 1733--1756.

%\bibitem{GLPR} 
%P. Grard, E.   Lenzmann,  O.  Pocovnicu, and P.   Rapha\"el.
%\newblock  A two-soliton with transient turbulent regime for the cubic half-wave equation on the real line.
%\newblock{\em   Ann. PDE} 4 (2018), no. 1, Paper No. 7, 166 pp.

  
  
\bibitem{Ho} 
T.L. Ho.
 \newblock   Bose-Einstein condensates with large number of vortices.
\newblock{\em  Physical review letters}, 87, no 6 (2001), 060403.
 
 
  \bibitem{Ianni-LeCoz}
I.    Ianni and S. Le Coz.
\newblock  Multi-speed solitary wave solutions for nonlinear Schrdinger systems. 
\newblock {\em J. Lond. Math. Soc.} (2) 89 (2014), no. 2, 623--639.


 \bibitem{KMR}
  J. Krieger, Y. Martel,  and P.  Rapha\"el. 
\newblock Two-soliton solutions to the three-dimensional gravitational Hartree equation. 
\newblock{\em Comm. Pure Appl. Math.} 62 (2009), no. 11, 1501--1550.

 


\bibitem{LCT}
S. Le Coz and T.-P. Tsai. 
\newblock  Finite and infinite soliton and kink-soliton trains of nonlinear Schr\"odinger equations. 
\newblock{\em Proceedings of the Sixth International Congress of Chinese Mathematicians. Vol. I, 43--56, Adv. Lect. Math. (ALM), 36, Int. Press, Somerville, MA, 2017.}
 
 
 
 \bibitem{Martel}
Y. Martel. 
\newblock  Interaction of solitons from the PDE point of view.
\newblock {\em Proceedings of the International Congress of Mathematicians--Rio de Janeiro 2018. Vol. III. Invited lectures, 2439--2466, World Sci. Publ., Hackensack, NJ, 2018.}
 
 


 

   \bibitem{Martel-Merle}
Y. Martel and F. Merle.
\newblock Multi solitary waves for nonlinear Schr\"odinger equations. 
\newblock {\em Ann. Inst. H. Poincar\'e Anal. Non Lin\'eaire} 23 (2006), no. 6, 849--864.
  
  
     \bibitem{Martel-Raphael}
Y. Martel and   P. Rapha\"el.
\newblock  Strongly interacting blow up bubbles for the mass critical nonlinear Schr\"odinger equation.
\newblock {\em  Ann. Sci. \'Ec. Norm. Sup\'er.} (4) 51 (2018), no. 3, 701--737. 


 
  
 \bibitem{Mueller-Ho} 
E. Mueller and T.-L. Ho.
 \newblock  Two-component Bose-Einstein condensates with a large number of vortices.
\newblock{\em   Phys. Rev. Lett.} Vol. 88, Iss. 18, 180403 (2002).


 \bibitem{Nier} 
 F. Nier. 
 \newblock  Bose-Einstein condensates in the lowest Landau level: Hamiltonian dynamics.
\newblock{\em  Rev. Math. Phys.} 19 (2007), no. 1, 101--130.
 
%\bibitem{OPoRo}
%B. Opanchuk, R. Polkinghorne, O. Fialko, J. Brand, and P. Drummond.
%\newblock Quantum simulations of the early universe.
%\newblock{\em Annalen der Physik.} Vol. 525 (2013),   Iss 10-11   Special Issue: SI   Pages: 866--876.
%
% 

% \bibitem{P2} 
% O. Pocovnicu.
% \newblock Explicit formula for the solution of the Szeg\H{o} equation on the real line and applications.
% \newblock{\em  Discrete Cont. Dyn. Syst.} 31(2011), 607--649.
 
 \bibitem{Schw-Tho}
V. Schwinte and L. Thomann.
\newblock Growth of Sobolev norms for  coupled  Lowest Landau Level equations. 
\newblock{\em  Pure Appl. Anal.} 3 (2021), no. 1, 189--222.


 
 
% \bibitem{Thirouin}
%J. Thirouin.
%\newblock Optimal bounds for the growth of Sobolev norms of solutions of a quadratic Szeg\H{o} equation. 
%\newblock{\em  Trans. Amer. Math. Soc.} 371 (2019), no. 5, 3673--3690.


 \bibitem{Thomann} 
L. Thomann.
 \newblock   Growth of Sobolev norms for  linear Schr\"odinger operators.
\newblock{\em Preprint : arXiv:2006.02674}.

 
%  \bibitem{Valet} 
%F.  Valet.
% \newblock  Asymptotic K-soliton-like Solutions of the Zakharov-Kuznetsov type.
%\newblock{\em Preprint : arXiv:2005.08518}.

  


% 
%
%
%
%
%\bibitem{Xu}
%H.  Xu.
%\newblock Unbounded Sobolev trajectories and modified scattering theory for a wave guide nonlinear Schr\"odinger equation. 
%\newblock{\em Math. Z.} 286 (2017), no. 1-2, 443--489. 
%
%
%\bibitem{Xu2}
%H.  Xu.
%\newblock Large-time blowup for a perturbation of the cubic Szeg\H{o} equation. 
%\newblock{\em Anal. PDE 7 (2014), no. 3}, 717--731.
%


\bibitem{Zhu}
K. Zhu. 
\newblock Analysis on Fock spaces.
\newblock{\em Graduate Texts in Mathematics}, 263. Springer, New York, 2012. x+344 pp.


\end{thebibliography}
\end{document}